\documentclass[11pt]{article}\textwidth 160mm\textheight 235mm
\usepackage{amsfonts}
\usepackage{amssymb}
\usepackage{graphicx}
\usepackage{amsmath}
\usepackage{amsthm}
\usepackage{enumitem}
\usepackage{tikz}
\usepackage{mathrsfs}
\usepackage{cancel}
\usepackage{todonotes}
\usepackage{cases}
\usetikzlibrary{matrix}
\usetikzlibrary{arrows,shapes}
\usepackage{cite}
\usepackage{bm}
\usepackage{hyperref}
\hypersetup{colorlinks=true, urlcolor=blue}
\expandafter\let\expandafter\oldproof\csname\string\proof\endcsname
\let\oldendproof\endproof
\renewenvironment{proof}[1][\proofname]{%
  \oldproof[\ttfamily \scshape \bf #1. ]%
}{\oldendproof}
\def\B{\mathbb{B}}
\def\R{{\rm I\!R}}
\def\oR{\overline{\R}}
\def\N{{\rm I\!N}}

\def\ox{\bar{x}}
\def\oy{\bar{y}}
\def\oz{\bar{z}}
\def\ov{\bar{v}}

\def\ou{\bar{u}}

\def \b{{\}_{k\in\N}}}

\def\xk{x^k}

\def\lmk{{\lambda^k}}

\def\what{\widehat}

\def\emp{\emptyset}
\def\tto{\rightrightarrows}

\def\prox{{\mbox{prox}\,}}

\def\Lm{{\Lambda}}
\def\tto{\rightrightarrows}

\def\sub{\partial}
\def\ra{\rangle}
\def\la{\langle}
\def\ve{\varepsilon}

\def\lm{\lambda}
\def\olm{{\bar\lambda}}

\def\gg{\gamma}
 \def\para{{\rm par}\,}
\def\dd{\delta}
\def\al{\alpha}
\def\Th{\Theta}

\def\ph{\varphi}

\def\vth{\vartheta}
\def\toset_#1{\xrightarrow{#1}}
\DeclareMathOperator*{\mini}{minimize\;}
\DeclareMathOperator*{\argmin}{argmin}
\def\d{{\rm d}}
\def\dist{{\rm dist}}
\def\spann{{\rm span}\,}
\def\rge{{\rm rge\,}}

\def\ri{{\rm ri}\,}

\def\inte{{\rm int}\,}
\def\gph{{\rm gph}\,}
\def\epi{{\rm epi}\,}
\def\dim{{\rm dim}\,}
\def\dom{{\rm dom}\,}

\def\ker{{\rm ker}\,}

\def\Kg{{K_g(\oz, \olm)}}

\def\K{{\overline{  K}}}
\def\rN{{\what{N}}}
\def\rs{{\what{\partial}}}
\def\rt{{\what{T}}}
\begin{document}
\vspace*{0.5in}
\begin{center}
{\bf A CHAIN RULE FOR  STRICT TWICE EPI-DIFFERENTIABILITY   AND ITS APPLICATIONS}\\[2ex]
NGUYEN T. V. HANG\footnote{  Institute of Mathematics, Vietnam Academy of Science and Technology, Hanoi, Vietnam (ntvhang@math.ac.vn). Research of this author is partially supported by Vietnam Academy of Science and Technology under the grant CTTH00.01/22-23.} and  M. EBRAHIM SARABI\footnote{Department of Mathematics, Miami University, Oxford, OH 45065, USA (sarabim@miamioh.edu). Research of this    author is partially supported by the U.S. National Science Foundation  under the grant DMS 2108546.}
\end{center}
\vspace*{0.05in}

\small{\bf Abstract.}  The presence of second-order smoothness for objective functions of optimization problems can provide valuable information about their stability properties 
and  help us design efficient numerical algorithms for solving these problems. Such second-order information, however, cannot  be expected in various 
 constrained  and composite optimization problems since we often have to express their objective functions in terms of  extended-real-valued functions for which  the classical second derivative may not exist. 
One powerful geometrical tool to use for dealing with such functions  is the concept of  twice epi-differentiability. In this  paper, we are going to study a stronger version of this concept, called strict twice epi-differentiability.
 We characterize this concept for certain composite functions and use it to establish the equivalence of metric regularity and strong metric regularity for a class of generalized 
   equations at their nondegenerate solutions. Finally, we present a characterization of continuous differentiability of the proximal mapping of our composite functions. 
 \\[1ex]

\newtheorem{Theorem}{Theorem}[section]
\newtheorem{Proposition}[Theorem]{Proposition}
\newtheorem{Lemma}[Theorem]{Lemma}
\newtheorem{Corollary}[Theorem]{Corollary}
\newtheorem{Definition}[Theorem]{Definition}
\numberwithin{equation}{section}

\theoremstyle{definition}
\newtheorem{Example}[Theorem]{Example}
\newtheorem{Remark}[Theorem]{Remark}

\renewcommand{\thefootnote}{\fnsymbol{footnote}}

\normalsize

\section{Introduction}

Twice epi-differentiability of extended-real-valued functions, introduced by Rockafellar in \cite{r89},
is a geometric concept that provides approximation of epigraphs and so differs from the classical quadratic approximation, obtained from the classical second derivative of functions. 
Such a geometrical approximation opens the door to deal with important classes of extended-real-valued functions, appearing in optimization problems.
Despite the fact that this concept was introduced in the late 1980s, major progress has been achieved recently in \cite{ms20,mms0} in understanding  its various properties.
{ Note that this concept plays via the notion of the second subderivative an important role in obtaining second-order optimality conditions for different classes of optimization problems; see \cite{bm, ms20, mms0,rw}.}
 It has also been shown recently that 
it can be an important tool for conducting convergence analysis of important numerical algorithms including the Newton method \cite{ms21} and the augmented Lagrangian method \cite{hs,r22}.
{\color{blue}
}
This paper aims to study  a stronger version of this concept, called {\em strict twice epi-differentiability},  for an extended-real-valued  function  $\ph:\R^n\to \oR$,  finite at $\ox\in \R^n$,  that enjoys the composite representation 
\begin{equation}\label{CF}
\ph (x)= (g \circ \Phi)(x)\quad  \mbox{for all}\;\;  x\in {  O}
\end{equation}
around $\ox$, where ${  O}$ is a neighborhood of $\ox$ and where  $\Phi:\R^n \to \R^m$ is twice continuously differentiable around $\ox$ and $g: \R^m \to \oR:=[-\infty,\infty]$ is  a polyhedral function.
Recall that a proper function $g: \R^m \to \oR$ is called polyhedral if its epigraph, namely the set $\epi g=\{(z,\al)\in \R^m\times \R|\; g(z)\le \al\}$, is  a polyhedral convex set. 
Traditionally,  when a function, satisfying the composite representation \eqref{CF},   enjoys a certain constraint qualification, it belongs to an important class of functions, called {\em fully amenable} (cf. \cite[Definition~10.23(a)]{rw}).
We, however, do not use this terminology since most of the results in this paper require a different constraint qualification than the one in \cite[Definition~10.23(a)]{rw}. 
Note also that the objective functions of  a number of  important classes of constrained and composite optimization problems, including 
classical nonlinear programming problems, constrained and unconstrained minimax problems, can be expressed in the composite form \eqref{CF}; see Example~\ref{point} for more details. 
This can pave the way to study various stability properties of these optimization problems and their KKT systems by exploring variational properties of the composite form in  \eqref{CF}.

{ Strict twice epi-differentiability  was first introduced in \cite{pr} and  explored further    in the context of  unconstrained optimization problems in \cite{pr2}, where it was  showed that 
it can characterize continuous differentiability of proximal mappings of prox-regular functions under certain assumptions. It was  studied then in  \cite{pr3} for nonlinear programming and minimax problems. 
Recently, the authors began exploring this concept for polyhedral functions in \cite{hjs} and showed that strict twice epi-differentiability in that setting amounts to a relative interior condition; see \cite[Theorem~4.3]{hjs}. 
In the context of optimization algorithms, strict twice epi-differentiability was utilized recently  in \cite{stp17,stp18} in order to design a linesearch algorithm for minimizing the sum of two nonconvex functions.}

One may wonder whether  strict twice epi-differentiability leads us to  stronger properties in comparison to those stemming from  its weaker version, namely twice epi-differentiability.   
As shown in Section~\ref{sec3}, this stronger version of twice epi-differentiability render the subgradient mapping of $\ph$ in \eqref{CF} to satisfy an interesting regularity property, called {\em strict proto-differentiability}.
Such a property, as demonstrated in Section~\ref{sec04}, has a number of important implications. First, it allows us to characterize the regularity of the graph of the subgradient mapping of $\ph$ in \eqref{CF}, which is 
a highly nonconvex set. Second, we are able to achieve  the equivalence of metric regularity and strong metric regularity for an important class of generalized equations at their nondegenerate solutions,
 which leads us to an extension of a similar result by Dontchev and Rockafellar in \cite{dr96}. 
 Using the latter  equivalence, we obtain a simple but useful characterization of  continuous differentiability of the proximal mapping of the composite function $\ph$
 in \eqref{CF}. Note that  while $\ph$ may not be a convex function,   it is well-known that its proximal mapping   is locally single-valued (cf. \cite[Proposition~13.37]{rw}) and directionally differentiable.
 Using the concept of strict proto-differentiability, we are going to show that one should expect a stronger differentiability property of the proximal mapping, namely continuous differentiability, when   certain subgradients of $\ph$ are chosen. 
{  Some of these properties, including continuous differentiability of proximal mappings, have been proven for  the ${\cal C}^2$-partly smooth functions defined by Lewis; see  \cite[Definition~4]{dhm} for the definition of such a function.
This suggests that the ${\cal C}^2$-partly smooth functions are likely strict twice epi-differentiable. 
 In fact, we will show in the follow-up paper \cite{hs23} that the class of strict  twice epi-differentable functions encompasses the ${\cal C}^2$-partly smooth functions. This opens a new door to study 
 the latter class of functions using tools of second-order generalized differentiation. }
 
{  Our primary goal is to extend the  characterization of strict twice epi-differentiability in \cite{hjs} to the composite function in \eqref{CF} and present a systematic approach 
for understanding this concept and its consequences in stability properties of generalized equations. To achieve this goal, one requires certain chain rules for different 
second-order generalized differentiation notions, utilized in \cite{hjs}. We begin by establishing 
 a new chain rule for  the strict second subderivative of \eqref{CF}. As shown in Remark~\ref{tilts}, such a chain rule can be exploited to study other stability properties of \eqref{CF}, a task we don't pursue in this paper. 
 We also justify a new chain rule for strict twice epi-differentiability of the composite function in \eqref{CF}. 
 We should add here that  the results in \cite{hjs} significantly benefit 
from   polyhedrality, assumed therein. Indeed, most of the proofs in \cite{hjs} relies one way or another upon the reduction lemma 
for polyhedral functions (cf. \cite[Theorem~3.1]{hjs}), which is not available for the composite form in \eqref{CF}. Using a different approach from \cite{hjs}, we also 
 establish the equivalence of metric regularity and strong metric regularity 
 for the class of generalized equations (see the generalized equation in \eqref{GE}) at their nondegenerate solutions, which allows us to characterize  continuous differentiability of the proximal mapping of \eqref{CF}. 
}
 
In Section~\ref{sec2}, we first recall important concepts, used in this paper, and then establish some elementary properties related to a constraint qualification that 
is going to be assumed in most of the results of this paper. Among these properties are certain metric estimates that allow us to achieve a chain rule for strict 
twice epi-differentiability. Section~\ref{sec3} begins with the concept of the strict second subderivative for which we will establish a chain rule. Using this result 
together with a chain rule for epi-convergence of functions, we achieve a characterization of strict twice epi-differentiability of $\ph$ in \eqref{CF} when an appropriate 
constraint qualification is assumed.  In Section~\ref{sec04}, we show that metric regularity and strong metric regularity are equivalent for a class of generalized equations at
their nondegenerate solutions. This equivalence is utilized next to obtain a characterization of continuous differentiability of the proximal mapping of the composite function in \eqref{CF}
as well as twice continuous differentiability of its Moreau envelope.

\section{Notation and Preliminary Results}\label{sec2}
In this paper,    we 
 denote by $\B$ the closed unit ball in the space in question and by $\B_r(x):=x+r\B$ the closed ball centered at $x$ with radius $r>0$. 
 Given a nonempty set $C\subset\R^n$, the symbols $C^*$, $\ri C$,  and $\spann C$ signify its polar cone, its relative interior,  and the linear space generated by  $C$, respectively. 
 For any set $C$ in $\R^n$, its indicator function $\dd_C$ is defined by $\dd_C(x)=0$ for $x\in C$ and $\dd_C(x)=\infty$ otherwise. We denote
 by $P_C$ the projection mapping onto $C$ and  by $\dist(x,C)$  the distance from $x\in \R^n$ to the set $C$.
 For a vector $w\in \R^n$, the subspace $\{tw |\, t\in \R\}$ is denoted by $[w]$. The domain and range of a set-valued mapping $F:\R^n\tto\R^m$ are defined, respectively, by  
$ \dom F= \{x\in\R^n\big|\;F(x)\ne\emp \}$ and $\rge F=\{u\in \R^m|\; \exists \,  w\in \R^n\;\,\mbox{with}\;\; u\in  F(w) \}$.    

Let $\{C^t\}_{t>0}$ be a parameterized family of sets in $\R^d$. Its inner and outer limit sets are defined, respectively,  by 
\begin{align*}
\liminf_{t\searrow 0} C^t&= \big\{x\in \R^d|\; \forall \; t_k \searrow 0 \;\exists \; x^{t_k}\to x \;\;\mbox{with}\;\; x^{t_k}\in C^{t_k}\; \; \mbox{for}\; \;  k\;\; \mbox{sufficiently large}\big\},\\
\limsup_{t\searrow 0} C^t&= \big\{x\in \R^d|\; \exists \; t_k \searrow 0 \;\exists\;   \; x^{t_k}\to x \;\;\mbox{with}\;\; x^{t_k}\in C^{t_k}\big\};
\end{align*}
see  \cite[Definition~4.1]{rw}. 
The limit set of $\{C^t\}_{t>0}$ exists if  $\liminf_{t\searrow 0} C^t=\limsup_{t\searrow 0} C^t =:C$,  written as $C^t \to C$ when $t\searrow 0$. 
 A sequence $\{f^k\b$ of functions $f^k:\R^n\to \oR$ is said to {\em epi-converge} to a function $f:\R^n\to \oR$ if we have $\epi f^k\to \epi f$ as $k\to \infty$;
 see \cite[Definition~7.1]{rw} for more details on the epi-convergence of 
a sequence of extended-real-valued functions. We denote by $f^k\xrightarrow{e} f$ the  epi-convergence of  $\{f^k\b$ to $f$.

Given a nonempty set $\Omega\subset\R^n$ with $\ox\in \Omega$, the tangent cone to $\Omega$ at $\ox$, denoted $T_\Omega(\ox)$,  is defined  by
\begin{equation*}\label{tan1}
T_\Omega(\ox) = \limsup_{t\searrow 0} \frac{\Omega - \ox}{t}.
\end{equation*}
The  regular/Fr\'{e}chet normal cone $\rN_\Omega(\ox)$ to $\Omega$ at $\ox$ is defined by
 $\rN_\Omega(\ox) = T_\Omega(\ox)^*$. For $x\notin \Omega$, we set $\rN_\Omega(x)=\emptyset$. 
 The limiting/Mordukhovich normal cone $N_\Omega(\ox)$ to $\Omega$ at $\ox$ is
 the collection of all vectors $\ov\in \R^n$ for which there exist sequences  $\{x^k\b$ and  $\{v^k\b$ with $v^k\in \rN_\Omega( x^k)$ such that 
$(x^k,v^k)\to (\ox,\ov)$. When $\Omega$ is convex, both normal cones boil down to that of convex analysis.  
Given a function $f:\R^n \to \oR$ and a point $\ox\in\R^n$ with $f(\ox)$ finite, the subderivative function $\d f(\ox)\colon\R^n\to\oR$ is defined by
\begin{equation*}\label{fsud}
\d f(\ox)(w)=\liminf_{\substack{
   t\searrow 0 \\
  w'\to w
  }} {\frac{f(\ox+tw')-f(\ox)}{t}}.
\end{equation*}
A vector $v\in \R^n$ is called a subgradient of $f$ at $\ox$ if $(v,-1)\in N_{\epi f}(\ox,f(\ox))$. The set of all subgradients of $f$ at $\ox$
is    denoted by $\sub f(\ox)$. Replacing the limiting normal cone with $\rN_{\epi f}(\ox,f(\ox))$ in the definition of $\sub f(\ox)$ gives us $\rs f(\ox)$, which is called the regular subdifferential 
of $f$ at $\ox$.  
The critical cone of $f$ at $\ox$ for $\bar v$ with $\bar v\in   \sub f(\ox)$ is defined by 
\begin{equation*}\label{cricone}
{K_f}(\ox,\bar v)=\big\{w\in \R^n\,\big|\,\la\bar v,w\ra=\d f(\ox)(w)\big\}.
\end{equation*}
When $f=\dd_\Omega$, where $\Omega$ is a nonempty subset of $\R^n$, the critical cone of $\dd_\Omega$ at $\ox$ for $\ov$ is denoted by $K_\Omega(\ox,\ov)$. In this case, the above definition of the critical cone of a function 
boils down to  the well known concept of critical cone to a set (see \cite[page~109]{DoR14}), namely $K_\Omega(\ox,\ov)=T_\Omega(\ox)\cap [\ov]^\perp$ because of  $\d \dd_\Omega(\ox)=\dd_{T_\Omega(\ox)}$.
 If $f$ is convex and $\partial f(\ox) \neq \emptyset$, then $\d f(\ox)$ is the support function of $\partial f(\ox)$ (cf. \cite[Theorem~8.30]{rw}) and the critical cone $K_f(\ox, \ov)$ can be equivalently described by
\begin{equation}\label{crif}
K_f(\ox, \ov) = N_{\partial f(\ox)}(\ov).
\end{equation}

 %According to \cite[Theorem~2.49]{rw},  a polyhedral function $g:\R^m\to \oR$ enjoys the  representation
%\begin{equation}\label{polyfunc}
%g(z) = \max_{j\in J}\big\{\la a^j, z \ra - \alpha_j\big\} + \delta_{\dom g}(z), \quad z\in \R^m,
%\end{equation}
%where $J:= \{1,\ldots, l\}$ for some $l\in \N$, $a^j\in \R^m$ and $\alpha_j\in \R$  for all $j\in J$, and where  $\dom g=\{z\in \R^m|\, g(z)<\infty\}$ is a polyhedral convex set with the representation 
%\begin{equation}\label{dom}
%\dom g = \big\{z\in\R^m\, \big|\, \la b^i, z \ra \leq \beta_i,  i\in I:=\{1, \ldots, s\} \big\}
%\end{equation}
%for some $s\in \N$,   $b^i \in \R^m$ and $\beta_i\in \R$ for all $i\in I$.  It is not hard to see that  $\dom g$ can be expressed as the finite  union of the polyhedral convex sets $C_j, j\in J$, defined by
%\begin{equation*}\label{Cj}
%C_j =\big\{z\in \dom g\, \big|\, g(z) = \la a^j, z\ra -\alpha_j\big\} =\big\{z \in \dom g\, \big|\, \la a^i -a^j, z\ra  \leq \alpha_i -\alpha_j, i\in J\big\}.
%\end{equation*}
%Pick $\oz \in\dom g$ and  define the sets of active indices  at $\oz$ corresponding to the representation \eqref{dom} and  to the partition  of $\dom g$ via the sets  $C_j $ by
%\begin{equation}\label{indset4}
%I(\oz)= \big\{i\in I\, \big|\, \la b^i, \oz\ra = \beta_i\big\}\quad\textrm{ and }\quad J(\oz)= \big\{ j\in J\, \big|\, \oz \in C_j\big\}.
%\end{equation}

Given a set-valued mapping $F:\R^n\tto \R^m$ and $(\ox,\oy)\in \gph F$, the graphical derivative of $F$ at $\ox$ for $\oy$, denoted by  $DF(\ox,\oy)$,  is a set-valued mapping
defined by 
\begin{equation}\label{proto}
\gph DF(\ox, \oy)=T_{\gph F}(\ox,\oy) = \limsup_{t\searrow 0}\frac{\gph F-(\ox,\oy)}{t}.
\end{equation}
 The coderivative of $F$ at   $\ox$ for $\oy$, denoted by  $D^* F(\ox,\oy) $, is 
defined via the relationship 
$$
w\in D^* F(\ox,\oy)(u)\iff (w,-u)\in N_{\gph F}(\ox,\oy).
$$
It is known (cf. \cite[Theorem~3.5(a)]{hjs} and \cite[Theorem~1]{ms16}, respectively) that   for a polyhedral function $g$ and $(\oz,\olm)\in \gph \sub g$, we always have 
\begin{equation}\label{cal2d}
\begin{cases}
D(\sub g)(\oz,\olm)(w)=N_{\Kg}(w)=\Kg^*\cap [w]^\perp\quad \mbox{for all}\;\; w\in \Kg,\\
  D^*(\sub g)(\oz,\olm)(0)=\Kg^*-\Kg^*=\spann\big\{ \Kg^*\big\}.
  \end{cases}
\end{equation}
Also, it follows from   \cite[Theorem~3.5(a)]{hjs} and \cite[Theorem~3.6(b)]{hjs}, respectively,  that 
\begin{equation}\label{domgd}
\begin{cases}
\dom D(\sub g)(\oz, \olm)=K_g(\oz, \olm)\quad \mbox{and}\\
 \dom D^*(\sub g)(\oz, \olm)=K_g(\oz, \olm)-K_g (\oz, \olm)=\spann\big\{\Kg\big\}.
 \end{cases}
\end{equation}

Take the  function $\ph$ from \eqref{CF} and   define the  Lagrange multiplier mapping $\Lm:\R^n\times \R^n\tto \R^m$  by  
\begin{equation}\label{laset}
\Lambda(x,v)=\big\{\lm\in\R^m\;\big|\;   \nabla \Phi(x)^*\lm=v,\;\lm\in \sub g(\Phi(x))\big\},\quad (x,v)\in\R^n\times \R^n.
\end{equation}
In the presence of an appropriate constraint qualification, we can ensure that $\Lambda(x,v)\neq \emptyset $ when   $(x,v)\in \gph \sub \ph$.
We are not concerned, however,  in this paper about what the weakest constraint qualification is to achieve this goal. 
To proceed, pick $\ov\in \sub \ph(\ox)$ and  suppose that $\olm\in \Lm(\ox,\ov)$. In what follows, we say that   the {\em second-order qualification condition} (SOQC)  is satisfied  at $(\ox,\olm)$ for the composite form \eqref{CF} if 
the condition  
\begin{equation}\label{soqc}
 D^*(\sub g)\big(\Phi(\ox),\olm\big)(0)\cap  \ker \nabla \Phi(\ox)^*=\{0\}
\end{equation}
holds. This condition has been often used in second-order variational analysis to obtain   second-order chain rules for important classes of functions including fully amenable functions; see \cite{mr,ms16}. 
While $D^*(\sub g)\big(\Phi(\ox),\olm\big)(0)$ was already calculated in \eqref{cal2d}, a more useful representation of it can be found in  \cite[Theorem~3.1(i)-(ii)]{ms16}. These results tell us that  the latter can be equivalently calculated as
\begin{equation}\label{code2}
D^*(\sub g)\big(\Phi(\ox),\olm\big)(0)=\para\{\sub g(\Phi(\ox))\},
\end{equation}
  where $\para \{\sub g(\Phi(\ox))\}$ stands for  the linear subspace of $\R^m$ parallel to the affine hull of $ \sub g(\Phi(\ox))$.
This shows that the SOQC \eqref{soqc} is equivalent to the transversality condition in the sense of \cite[Definition~3.4]{be}. Moreover, using \eqref{code2}, one can see that the SOQC is equivalent 
to the nondegeneracy condition for the composite form \eqref{CF}, defined in \cite[Definition~5.3.1]{mi}. To elaborate more on the SOQC \eqref{soqc}, we consider two special cases of 
the polyhedral function $g$ and then show that the latter condition, indeed, boils down to well known conditions in optimization.

\begin{Example} \label{point} Suppose that $g$  and $\Phi$ are taken from \eqref{CF} and  $\Phi=(\ph_1,\ldots,\ph_m)$ with $\ph_i:\R^n\to \R$ for all $i=1,\ldots,m$.
We are going to discuss the SOQC \eqref{soqc} for the following instances of the polyhedral function $g$:
\begin{itemize}[noitemsep,topsep=2pt]
\item [ \rm {(a)}]   Suppose that $g(z)=\max\{z_1,\ldots,z_m\}$ with $z=(z_1,\ldots,z_m)\in \R^m$, $\oz=\Phi(\ox)$. 
Assume without loss of generality that $\oz_1=\ldots=\oz_m$, where $\oz=(\oz_1,\ldots,\oz_m)$. 
This selection of $g$
enables us to cover a minimax problem.
In this case, it is not hard to see that  the SOQC \eqref{soqc} amounts to saying that 
the vectors   $\nabla\ph_1(\ox), \ldots, \nabla \ph_m(\ox)$  are affinely independent; see \cite[Proposition~3]{ms16} and its proof for more detail. Recall that a set of vectors $w_0,w_1,\ldots,w_m$ is affinely independent if the vectors $ w_1-w_0,\ldots,w_m-w_0$
are linearly independent. 

%To justify our claim about the SOQC \eqref{soqc} in this setting, assume without loss of generality that $J(\oz)=\{1,\ldots,m\}$. Thus, we deduce from \eqref{code2} that 
%$$
%D^*(\sub g)\big(\Phi(\ox),\olm\big)(0)= \spann\big\{ e^i-e^j\, | \, i,j\in J(\oz)\big\}=\spann\big\{ e^i-e^1\, | \, i\in J(\oz)\big\},
%$$
%where the vectors $e^i$, $i=1,\ldots,m$, are unit vectors in $\R^m$ that have 1 as the $i^{th}$ component and zeros elsewhere.   
%This clearly tells us that the SOQC \eqref{soqc}
%amounts to the set $\{ \nabla\ph_2(\ox)-\nabla\ph_1(\ox),\ldots,\nabla\ph_m(\ox)-\nabla\ph_1(\ox)\}$ being linearly independent and thus  confirms our claim.

\item [ \rm {(b)}]  Suppose that  $g=\dd_C$,  where $C=\{0\}^{s}\times \R^{m-s}_-$ with $0\le s\le m$. This selection of $g$ allows us to
  cover both equality and inequality constraints appearing in classical nonlinear programming problems. 
 In this case, it is not hard to see that the SOQC \eqref{soqc} reduces to the classical linear independence constraint qualification (LICQ);
see the discussion after \cite[Definition~6.1]{msr} for a proof of this fact.
\end{itemize}
\end{Example}

We proceed with some direct consequences of the SOQC \eqref{soqc} for the composite form \eqref{CF}, which play an important role for our developments in the next section. 
 
\begin{Proposition} \label{soqcg}Assume that $\ph:\R^n\to \oR$ has the representation \eqref{CF} around $\ox\in \R^n$, $\ov\in \sub \ph(\ox)$, and that $\olm\in \Lm(\ox,\ov)$. If the SOQC \eqref{soqc} is satisfied  at $(\ox,\olm)$, then 
 the following properties hold.
\begin{itemize}[noitemsep,topsep=2pt]
\item [ \rm {(a)}] The basic constraint qualification 
\begin{equation}\label{bcq}
N_{\dom g}(\Phi(x))  \cap  \ker \nabla \Phi(x)^*=\{0\}
\end{equation}
holds for any $x$ close to $\ox$ with $\Phi(x)\in \dom g$. 
\item [ \rm {(b)}]   There exist a constant $\ell\ge 0$ and  a neighborhood $U$ of $(\ox,\ov)$ such that for any $(x,v)\in U\cap \gph \sub \ph$, 
 the Lagrange multiplier set $\Lm(x,v)$ is a singleton and the estimate 
\begin{equation}\label{mslag}
\|\lm-\olm\|\le \ell \big(\|x-\ox\|+\|v-\ov\|\big)
\end{equation}
holds, where $\lm $ is the unique Lagrange multiplier in $\Lm(x,v)$. 
\end{itemize}
\end{Proposition}
\begin{proof} 
By \eqref{cal2d}, we obtain from \eqref{soqc} that  
\begin{equation}\label{dc}
 D(\sub g)\big(\Phi(\ox),\olm\big)(0)\cap  \ker \nabla \Phi(\ox)^*=\{0\}.
\end{equation}
Appealing now to \cite[Proposition~3.10]{s20} illustrates that this property is equivalent to the condition  $\Lm(\ox,\ov)=\{\olm\}$.
It follows from \cite[Proposition~10.21]{rw} that  $\dom \d g(\Phi(\ox))=T_{\dom g}(\Phi(\ox))$, which allows us to conclude  the inclusion $K_g(\Phi(\ox),\olm)\subset T_{\dom g}(\Phi(\ox))$. This in turn gives us  $N_{\dom g}(\Phi(\ox))\subset K_g(\Phi(\ox),\olm)^*$. 
 Combining this, \eqref{cal2d}, and \eqref{dc} leads us to 
\begin{equation}\label{bcq1}
N_{\dom g}(\Phi(\ox))  \cap  \ker \nabla \Phi(\ox)^*=\{0\}.
\end{equation}
This tells us that \eqref{bcq} is satisfied  for any $x$ close to $\ox$ with $\Phi(x)\in \dom g$. 
If not, we find  sequences $x^k\to \ox$ with $\Phi(x^k)\in \dom g$ and  $\{\eta^k\b\subset \R^m$ with $\eta^k\neq 0$ such that 
$\eta^k\in N_{\dom g}(\Phi(x^k))  \cap  \ker \nabla \Phi(x^k)^*$. Assume without loss of generality that $\eta^k/\|\eta^k\|\to q$ with $q\in \R^m\setminus\{0\}$. 
Thus, we get $q\in N_{\dom g}(\Phi(\ox))  \cap  \ker \nabla \Phi(\ox)^*$, which contradicts \eqref{bcq1} and hence proves (a). 

{  Turning now to (b), we deduce from (a) that 
  there exists a $\dd>0$ such that for all $(x,v)\in \big(\gph \sub \ph\big)\cap \B_\dd(\ox,\ov)$, the Lagrange multiplier set $\Lm(x,v)$ is nonempty and uniformly bounded. 
  The nonemptiness results  from \eqref{bcq} and \cite[Example~10.8]{rw}. 
  To justify the claim about uniform boundedness, suppose by contradiction that there are sequences $(x^k,v^k)\to (\ox,\ov)$  and $\lm^k\in \Lm(x^k,v^k)$ such that $\|\lm^k\|\to \infty$. 
  Assume again without loss of generality that $\lm^k/\|\lm^k\|\to q$ with $q\in \R^m\setminus\{0\}$. By \eqref{laset}, we get 
  $q\in \ker \nabla\Phi(\ox)^* $ and $\lm^k \in \sub g(\Phi(x^k))$. The latter tells us that $(\lm^k,-1)\in N_{\epi g}\big(\Phi(x^k),g(\Phi(x^k))\big)$. 
  Since $g$ is continuous relative to its domain, dividing by $\|\lm^k\|$ and passing then to the limit in the latter inclusion bring us to $(q,0)\in N_{\epi f}\big(\Phi(\ox),g(\Phi(\ox))\big)$. 
  Appealing to the definition of the normal cone leads us to $q\in N_{\dom g}(\Phi(\ox))$, which coupled with $q\in \ker \nabla\Phi(\ox)^* $, violates \eqref{bcq1}
  and hence  proves the uniform boundedness of $\Lm(x,v)$. }
  
Next, we claim  that there exists $\dd'\in (0, \dd)$ such that for all $(x,v)\in \big(\gph \sub \ph\big)\cap \B_{\dd'}(\ox,\ov)$
and all $\lm\in \Lm(x,v)$, the condition 
\begin{equation*}\label{socq2}
D^*(\sub g)(\Phi(x), \lm)(0)\cap \ker \nabla \Phi(x)^*=\{0\}
\end{equation*}
holds. Suppose to the contrary that for any $k\in \N$, there are a pair  $(x^k,v^k)\in \gph \sub \ph$, converging to $(\ox,\ov)$, and 
$\lm^k\in \Lm(x^k,v^k)$ for which we can find $\eta^k\in \R^m\setminus\{0\}$ such that 
\begin{equation}\label{eq05}
\eta^k\in D^*(\sub g )(\Phi(x^k), \lm^k)(0)\cap \ker \nabla \Phi(x^k)^*.
\end{equation}
According to the discussion above,  the Lagrange multiplier sets $\Lm(x^k,v^k)$ are  uniformly bounded for sufficiently large $k$. 
Passing to a subsequence, if necessary, we can assume that the sequence $\{\lm^k\b$ converges to a vector in $\Lm(\ox,\ov)$.
Since  $\Lm(\ox,\ov)=\{\olm\}$,   we arrive at $\lm^k\to \olm$.
Passing to a subsequence again if necessary, we can assume that $\eta^k/\|\eta^k\|\to q$ with  $q\in \R^m\setminus\{0\}$.
Combining these and using \eqref{eq05} bring us to
$$
q\in D^*(\sub g)(\Phi(\ox), \olm)(0)\cap \ker \nabla \Phi(\ox)^*,
$$
which contradicts \eqref{soqc} since $q\neq 0$ and thus proves our claim. Similar to the argument presented above  for \eqref{dc}, 
we get for any $(x,v)\in \big(\gph \sub \ph\big)\cap \B_{\dd'}(\ox,\ov)$  
and any $\lm\in \Lm(x,v)$ that   the condition 
\begin{equation}\label{dc2} 
D(\sub g)(\Phi(x), \lm)(0)\cap \ker \nabla \Phi(x)^*=\{0\}
\end{equation}
is satisfied and that $\Lm(x,v)$ is a singleton. 
To justify the estimate \eqref{mslag}, suppose by contradiction that it fails. Thus, for any $k\in \N$, we find $(x^k,v^k)\in \gph \sub \ph$, converging to $(\ox,\ov)$, and 
$\lm^k\in \Lm(x^k,v^k)$ such that 
$$
\|\lm^k-\olm\|> k(\|x^k-\ox\|+\|v^k-\ov\|).
$$
Setting $t_k:=\|\lm^k-\olm\|$, we get $\|x^k-\ox\|=o(t_k)$ and $\|v^k-\ov\|=o(t_k)$. Similar to the argument utilized above, we can assume by passing to a subsequence if necessary that $\lm^k\to \olm$.
By the definitions of $\lm^k$ and $\olm$, we obtain 
$$
\nabla\Phi(\ox)^*\big(\frac{\lm^k-\olm}{t_k}\big)=\frac{1}{t_k}\big((\nabla\Phi(\ox)-\nabla\Phi(x^k))^*\lm^k+ v^k-\ov\big)=\frac{o(t_k)}{t_k}.
$$
Assume without loss of generality that $(\lm^k-\olm)/t_k\to q$ for some $q\in \R^m\setminus\{0\}$. The latter immediately implies that $q\in \ker\nabla \Phi(\ox)^*$. 
On the other hand, since $\sub g(\Phi(\ox))$ is a polyhedral convex set and since $\lm^k\in \sub g(\Phi(x^k))\subset \sub g(\Phi(\ox))$, it follows from \cite[Exercise~6.47]{rw} and \eqref{crif} that $$\lm^k-\olm\in T_{\sub g(\Phi(\ox))}(\olm)=(N_{\sub g(\Phi(\ox))}(\olm))^*=K_g(\Phi(\ox),\olm)^*$$
for all $k$ sufficiently large. By \eqref{cal2d}, we have $D(\sub g)\big(\Phi(\ox),\olm\big)(0)=K_g(\Phi(\ox),\olm)^*$.
Thus, we get  $q\in  D(\sub g)\big(\Phi(\ox),\olm\big)(0)$, a contradiction with \eqref{dc}. This proves (b) and completes the proof. 
  \end{proof}
  
 We continue with another consequence  of the SOQC \eqref{soqc}, important for the calculus of strict second subderivative in the next section.
 To this end, recall that a set-valued mapping   $F:\R^n \tto \R^m$ is said to be {\em metrically regular} at $\ox$ for $\oy\in F(\ox)$ if   there exist $\kappa \geq 0$  and  neighborhoods $U$ of $\ox$ and $V$ of $\oy$ such that the distance estimate
\begin{equation}\label{mr8}
\dist \big(x, F^{-1}(y)\big)\leq \kappa\, \dist\big(y, F(x)\big)
\end{equation}
holds for all $(x, y)\in U\times V$. When the estimate in \eqref{mr8} holds for any $(x,y)\in \R^n\times \R^m$, we say that $F$ is {globally} metrically regular. 
Recall that $F$ is said to be positively  homogeneous if $0\in F(0)$ and $F(tx)=tF(x)$ for any $t>0$ and $x\in \R^n$.
It is known (cf. \cite[Theorem~5.9(a)]{io}) that when $F$ is positively  homogeneous and metrically regular at $0\in \R^n$ for $0\in \R^m$, it is globally metrically regular. 

Given a closed set $C\subset \R^m$, a continuously differentiable function $h:\R^n\to \R^m$, and $\ox\in \R^n$ with $h(\ox)\in C$, it follows from \cite[Example~9.44]{rw} that the mapping $x\mapsto h(x)-C$
is metrically regular at $\ox$ for $0$ if and only if the condition 
$$
N_C(h(\ox))\cap \ker\nabla h(\ox)^*=\{0\}
$$
 is satisfied. In this case, the infimum of the constants $\kappa$ for which   metric regularity  of the mapping $x\mapsto h(x)-C$ holds at $\ox$ for $0$  can be calculated by 
 \begin{equation}\label{infk}
 \max\Big\{\frac{1}{\|\nabla h(\ox)^*y\|}\,\big |\; y\in N_C(h(\ox)), \, \|y\|=1\Big\}.
 \end{equation}
 
 \begin{Proposition} \label{mrcon}
 Assume that $\ph:\R^n\to \oR$ has the representation \eqref{CF} around $\ox\in \R^n$, $\ov\in \sub \ph(\ox)$, and that $\olm\in \Lm(\ox,\ov)$. If the SOQC \eqref{soqc} holds at $(\ox,\olm)$, then
 there is a  neighborhood $U$ of $(\ox,\ov)$ such that for any $(x,v)\in (\gph \sub \ph )\cap U$, the mapping $w\mapsto \nabla \Phi(x)w-K_g(\Phi(x),\lm)$ is globally metrically regular with constant $2\bar\gg$,
where $\lm$ is the unique Lagrange multiplier in $\Lm(x,v)$ and where 
\begin{equation}\label{numell}
\bar\gg=\max\Big\{\frac{1}{\|\nabla \Phi(\ox)^*y\|}\,\big|\; \|y\|=1, y\in D^*(\sub g)(\Phi(\ox),\olm)(0)\Big\}.
\end{equation}
 \end{Proposition}
  \begin{proof} Following the proof of Proposition~\ref{soqcg}, we find $\delta>0$ such that for any $(x,v)\in (\gph \sub \ph )\cap \B_\dd(\ox,\ov)$
  the condition \eqref{dc2} holds.  Shrinking $\delta$ if necessary, we conclude  from  Proposition~\ref{soqcg}(b) that $\Lm(x,v)$ is a singleton for any $(x,v)\in (\gph \sub \ph )\cap \B_\dd(\ox,\ov)$. 
Take any such a pair $(x,v) $,  let $\lm$ be the unique Lagrange multiplier in $\Lm(x,v)$ and define the mapping $G_{x,\lm}:\R^n\tto \R^m$ by $G_{x,\lm}(w)= \nabla \Phi(x)w-K_g(\Phi(x),\lm)$ for any $w\in \R^n$. 
It is not hard to check that $G_{x,\lm}$ is positively  homogeneous. Moreover, it follows from \eqref{cal2d} that $D(\sub g)(\Phi(x), \lm)(0)=N_{K_g(\Phi(x),\, \lm)}(0)$. 
 Appealing now to \eqref{dc2}, we arrive at 
\begin{equation}\label{mrc}
N_{K_g(\Phi(x), \lm)}(0)\cap \ker \nabla \Phi(x)^*=\{0\}.
\end{equation}
 By the discussion prior to this proposition, the latter condition 
amounts to metric regularity of the mapping $G_{x,\lm}$   at $0\in \R^n$ for $0\in \R^m$. As pointed out above, it follows from  \cite[Example~9.44]{rw} that 
 the infimum of the constant $\gamma_{x,\lm}$
for which   metric regularity of  $G_{x,\lm}$ holds at $0$ for $0$   can be calculated by 
$$
\gamma_{x,\lm}:=\max\Big\{\frac{1}{\|\nabla \Phi(x)^*y\|}\, \big |\; \|y\|=1, y\in N_{K_g(\Phi(x),\,\lm)}(0)\Big\}.
$$
We claim that there exists $r\in (0,\delta)$ such that  for any $(x,v)\in (\gph \sub \ph )\cap \B_r(\ox,\ov)$, we have $\gamma_{x,\lm}\le 3\bar\gg/2$ with 
$\bar\gg$ taken from \eqref{numell}. To justify the claim, suppose by contradiction that for any $k\in \N$,  there exist $(x^k,v^k)\in \gph \sub \ph$ with $(x^k,v^k)\to (\ox,\ov)$,   
and $y^k\in N_{K_g(\Phi(x^k),\lm^k)}(0)$ with $\|y^k\|=1$ and $\lm^k\in \Lambda(x^k,v^k)$ such that 
$$
\frac{1}{\|\nabla \Phi(x^k)^*y^k\|} >\frac{3\bar\gg}{2}.
$$
Passing to a subsequence if necessary, we can assume that 
$y^k\to y$ for some vector $y\in \R^m$ with  $\|y\|=1$. Moreover,  we have 
\begin{equation*}
y^k\in    N_{K_g(\Phi(x^k),\,\lm^k)}(0)= K_g(\Phi(x^k),\lm^k)^*
\subset \spann\{ K_g(\Phi(x^k),\lm^k)^*\}= D^*(\sub g)(\Phi(x^k),\lm^k)(0),
\end{equation*}
where the last equality results from \eqref{cal2d}. 
By Proposition~\ref{soqcg}(b), we obtain   $\lm^k\to \olm$.
Passing to the limit then brings us to 
$$
\frac{1}{\|\nabla \Phi(\ox)^*y\|} \ge \frac{3\bar\gg}{2}, \;\; y\in D^*(\sub g)(\Phi(\ox),\olm)(0), \;\; \|y\|=1,
$$
which is a contradiction with the definition of $\bar\gg$. This proves our claim and thus indicates that  there exists $r\in (0,\delta)$ such that  for any $(x,v)\in (\gph \sub \ph )\cap \B_r(\ox,\ov)$,
the infimum of the constants  $\gamma_{x,\lm}$ of metric regularity of $G_{x,\lm}$ at $0$ for $0$  is strictly less than $2\bar \gg$. This tells us that $2\bar\gg$ can be chosen as 
a common constant of metric regularity of $G_{x,\lm}$ at $0$ for $0$ for any    $(x,v)\in (\gph \sub \ph )\cap \B_r(\ox,\ov)$. Since $G_{x,\lm}$ are positively homogenous, it results from 
\cite[Theorem~5.9(a)]{io} that $G_{x,\lm}$ are globally metrically regular with the same constant $2\bar\gg$, which completes the proof.
  \end{proof}
  
\begin{Proposition} \label{redu}Assume that $g:\R^m\to \oR$ is a polyhedral function, $\oz\in \dom g$, and that $s=\dim (\para\{\sub g(\oz)\})$. Then there exist a neighborhood $U$ of $\oz$,  
 an $s\times m$ matrix $B$, and a polyhedral function $\vartheta:\R^s\to \oR$  for which we have $g(z)=\vartheta(Bz)$ for any $z\in U$ and $\ker B=(\para\{\sub g(\oz)\})^\perp$. Consequently, $B$ has full rank. 
 \end{Proposition}
\begin{proof} The first claim was established in \cite[Lemma~3.1]{ms16}. The second claim was justified in the proof of \cite[Theorem~3.2]{ms16}.
Finally, the last claim is a direct consequence of the classical rank-nullity theorem from linear algebra. 
\end{proof}

The above result has an important implication for the composite form \eqref{CF} in the presence of the SOQC \eqref{soqc}, allowing us to equivalently express 
$\ph$ in the form of \eqref{CF} with the SOQC replaced by a stronger condition.
\begin{Corollary} \label{redu2} Assume that $\ph:\R^n\to \oR$ has the representation \eqref{CF} around $\ox\in \R^n$, $\ov\in \sub \ph(\ox)$, and that $\olm\in \Lm(\ox,\ov)$ and  the SOQC \eqref{soqc} holds at $(\ox,\olm)$. Then
there exists a neighborhood $O$ of $\ox$ on which $\ph$ can be expressed as 
 \begin{equation}\label{CFn}
 \ph(x)=(\vartheta\circ \Psi)(x)\quad \mbox{for all}\;\; x\in O,
 \end{equation}
where $\vartheta:\R^s\to \oR$ is a polyhedral function with $s=\dim (\para\{\sub g(\Phi(\ox))\})$ and where $\Psi:\R^n\to \R^s$, defined by $\Psi=B\circ \Phi$ with $B$  and $\Phi$ taken from Proposition~{\rm\ref{redu}}
and \eqref{CF}, respectively. Moreover,  $\nabla \Psi(\ox)$ has full rank.
\end{Corollary}
\begin{proof} Shrinking the neighborhood $O$ in \eqref{CF}, together with Proposition~\ref{redu}, gives us the representation \eqref{CFn}. 
The claim about $\nabla \Psi(\ox)$ was   established  in \cite[Proposition~4.1]{ms16}. 
\end{proof}
    
    A function  $f\colon\R^n\to\oR$ is called {  prox-regular} at $\ox$ for $\ov$ if $f$ is finite at $\ox$ and locally lower semicontinuous (lsc)  around $\ox$ with $\ov\in\sub f(\ox)$, and there exist 
constants $\ve>0$ and $\rho> 0$ such that for all $x\in\B_{\ve}(\ox)$ with $f(x)\le f(\ox)+\ve$ we have
\begin{equation}\label{prox}
f(x)\ge f(u)+\la v,x-u\ra-\frac{\rho}{2}\|x-u\|^2\;\mbox{ whenever }\;(u,v)\in(\gph\sub f)\cap\B_{\ve}(\ox,\ov).
\end{equation}
The function $f$ is called {  subdifferentially continuous} at $\ox$ for $\ov$ if the convergence $(x^k,v^k)\to(\ox,\ov)$ with $v^k\in\sub f(x^k)$ yields $f(x^k)\to f(\ox)$ as $k\to\infty$. 
  
  \begin{Proposition}\label{phpro}
  Assume that $\ph$ has the representation \eqref{CF} around $\ox\in \R^n$ and  $\ov\in \sub \ph(\ox)$, and that the basic constraint qualification \eqref{bcq1} holds at $\ox$.
  Then the following properties hold.
  \begin{itemize}[noitemsep,topsep=2pt]
\item [ \rm {(a)}]  $\ph$ is prox-regular and subdifferentially continuous at $\ox$ for   $\ov$.
\item [ \rm {(b)}] The critical cone $K_{\ph}(\ox,\ov)$ can be represented by 
\begin{equation}\label{criph}
K_\ph(\ox,\ov)=\big\{w\in \R^n|\; \nabla \Phi(\ox)w\in K_g(\Phi(\ox), \olm)\big\}=N_{\sub\ph(\ox)}(\ov),
\end{equation}
where $\olm$ can be any vector in $\Lm(\ox,\ov)$. Moreover, for any $w\in K_{\ph}(\ox,\ov)$, we always have 
$$
N_{K_\ph(\ox,\ov)}(w)=\nabla\Phi(\ox)^*N_{K_g(\Phi(\ox), \olm)}( \nabla\Phi(\ox)w).
$$
\item [ \rm {(c)}] We have $\sub \ph(\ox)=\rs \ph(\ox)$. Consequently, $N_{\sub \ph(\ox)}(\ov)$ is a linear subspace if and only if $\ov\in \ri \sub \ph(\ox)$.
\item [ \rm {(d)}] $\ov\in \ri \sub \ph(\ox)$ if and only if $\olm\in \ri \sub g(\Phi(\ox))$ for any $\olm\in \Lm(\ox,\ov)$.
\end{itemize}
  \end{Proposition}
 \begin{proof} Part (a) results from     \cite[Proposition~13.32]{rw}.  The first equivalent representation of $K_\ph(\ox,\ov)$  in (b) follows from  the chain rule for the subderivative in \cite[Theorem~13.14]{rw},
 which holds under the  basic constraint qualification \eqref{bcq1}. The second one was taken from \cite[Theorem~13.14]{rw}.
The claimed chain rule for normal cones in (b) is an immediate consequence of the fact that  both critical cones $K_\ph(\ox,\ov)$ and $K_g(\Phi(\ox), \olm)$ are polyhedral.
The first claim in (c) can be found in \cite[Exercise~10.25(a)]{rw}. The second claim results from the fact that $\sub \ph(\ox)$ is  convex.
To justify (d), observe that \eqref{bcq1} ensures that $\Lm(\ox,\ov)\neq \emptyset$. Thus, for any $\olm\in \Lm(\ox,\ov)$, we have $\ov=\nabla\Phi(\ox)^*\olm$. The claimed equivalence then results from \cite[Proposition~2.44(a)]{rw}.
 \end{proof}

\section{Chain Rule for Strict Twice Epi-Differentiability}\label{sec3}

This section is devoted to study an important second-order variational property, called strict twice epi-differentiability, for  the composite function $\ph$ in \eqref{CF}.
To this end, consider  a function $f: \R^n \to \oR$ and $\ox \in \R^n$ with $f(\ox)$ finite and  define the parametric  family of 
second-order difference quotients of $f$ at $\ox$ for $\ov\in \sub f(\ox)$ by 
$$
\Delta_t^2 f(\bar x , \ov)(w)=\frac{f(\ox+tw)-f(\ox)-t\langle \ov,\,w\rangle}{\frac {1}{2}t^2}
$$
for any $w\in \R^n$ and $t>0$. The {second subderivative} of $f$ at $\ox$ for $\ov$, denoted $\d^2 f(\bar x , \ov)$, is an extended-real-valued function defined  by 
\begin{equation*}\label{ssd}
\d^2 f(\bar x , \ov)(w)= \liminf_{\substack{
   t\searrow 0 \\
  w'\to w
  }} \Delta_t^2 f(\ox , \ov)(w'),\;\; w\in \R^n.
\end{equation*}
Following \cite[Definition~13.6]{rw},   $f$ is said to be {twice epi-differentiable} at $\bar x$ for $\ov$
if  the functions $  \Delta_t^2 f(\bar x , \ov)$ epi-converge to $  \d^2 f(\bar x,\ov)$ as $t\searrow 0$.  Further, we say that $f$ is {\em strictly} twice epi-differentiable at $\bar x$ for $\ov$ if 
the functions $  \Delta_t^2 f(  x , v)$ epi-converge  to a function as $t\searrow 0$, $(x,v)\to (\ox,\ov)$ with $f(x)\to f(\ox)$ and $(x,v)\in \gph \sub f$.  If this condition holds, the limit function is then the second subderivative $  \d^2 f(\bar x,\ov)$.
The following result, established recently in \cite[Theorem~4.3]{hjs},
achieved a simple characterization of strict twice epi-differentiability of polyhedral functions.

 \begin{Proposition}[strict twice epi-differentiability of polyhedral functions]\label{riste}
Assume that $g:\R^m\to \oR$ is a polyhedral function and  that $(\oz,\olm)\in \gph \sub g$. Then 
 the following  properties are equivalent:
\begin{itemize}[noitemsep,topsep=2pt]
\item [ \rm {(a)}]  there is a neighborhood $U$ of $(\oz,\olm)$ such that for any $(z,\lm)\in U\cap \gph \sub g$, $g$ is strictly twice epi-differentiable at $z$ for $\lm$;
\item [ \rm {(b)}] $\olm \in \ri \partial g(\oz)$.
\end{itemize}
\end{Proposition} 

It is insightful to add that  the main driving force of  the above characterization  was  the following   result for polyhedral functions,  observed in \cite[Corollary~3.4(b)]{hjs}, which holds under 
 the relative interior condition $\olm \in \ri \partial g(\oz)$:
There exists a neighborhood $V$  of $(\oz,\olm)$  for which we have 
\begin{equation}\label{crieq}
K_g(z,\lm)=K_g(\oz,\olm)\quad \mbox{for all}\;\; (z,\lm)\in V\cap \gph \sub g.
\end{equation}

The above result  motivates us to pursue the possibility of obtaining a similar characterization of strict twice epi-differentiability 
for the composite form \eqref{CF}.  To this end, we define the {\em strict second subderivative} of $f$ at 
$\bar x$ for $\ov$ with  $\ov \in \partial f(\ox)$ at $w\in \R^n$ by 
\begin{equation*}
\d_s^2f(\ox, \ov)(w) = \liminf_{\substack{
t\searrow 0, \, w'\to w \\
(x, v) \toset_{\gph \partial f}(\ox, \ov)\\
f(x)\to f(\ox)}} 
\Delta_t^2 f(x , v)(w').
\end{equation*}
When $f$ is subdifferentially continuous at $\ox$ for $\ov$,  we can drop the requirement $f(x)\to f(\ox)$ in the definition above. 
Two immediate observations   can be made about  the strict second subderivative. The first one is that it is positively homogeneous of degree $2$ (see \cite[Definition~13.4]{rw} for its definition). 
The second one is that it is always lsc. Both claims can be proven as of those for the second subderivative established in \cite[Proposition~13.5]{rw}. 
\begin{Lemma} \label{prof}Assume that $f:\R^n\to \oR$    and $(\ox,\ov)\in \gph \sub f$. Then we have
$$
\d_s^2f(\ox, \ov)(w) \le \liminf_{\substack{
w'\to w \\
(x, v) \toset_{\gph \partial f}(\ox, \ov)\\
f(x)\to f(\ox)}} 
\d^2 f(x , v)(w').
$$
\end{Lemma}
\begin{proof}
Denote by $\al$ the value of the limit on the right-hand side of the above inequality. Thus, we can find sequences $(x^k,v^k)\to (\ox,\ov)$ with  $(x^k,v^k)\in \gph \sub f$, $f(x^k)\to f(\ox)$,  and $w^k\to w$
for which we have $  \d^2f(x^k,v^k)(w^k)\to \al $. By the definition of the second subderivative, for each $k$,  we find sequences $w^k_m\to w^k $ and $t^m_k\searrow 0$ as $m\to \infty$
such that $\Delta^2_{t_k^m}f(x^k,v^k)(w^k_m)\to \d^2f(x^k,v^k)(w^k)$ as $m\to \infty$. Using a standard diagonalization technique, we can find sequences $m_k\to \infty$, $w^k_{m_k}\to w$ and $t_k^{m_k}\searrow 0$
as $k\to \infty$ such that $\Delta^2_{t_k^{m_k}} f(x^{k},v^{k})(w^k_{m_k})\to  \al$ as $k\to \infty$. It is evident from the definition of the strict second subderivative that $\al \ge \d_s^2f(\ox,\ov)(w)$, which proves  our claimed inequality.
\end{proof}
 We show below that the strict second subderivative can be utilized to obtain a characterization of the so-called {\em uniform quadratic growth condition};
 see   \cite[Definition~5.16]{bs}. The latter condition is known to be equivalent to the {\em   variational strong convexity}, defined by Rockafellar in \cite[Definition~2]{r17}, for  lsc functions; see Corollary~\ref{vsc} for the definition of this property. 
 
 \begin{Proposition}\label{uqgc} Assume that $f:\R^n\to \oR$ is an lsc function  and $(\ox,\ov)\in \gph \sub f$. Then the following properties are equivalent:
 \begin{itemize}[noitemsep,topsep=2pt]
\item [ \rm {(a)}]  $\d^2_s f(\ox,\ov)(w)>0$ for all $w\in \R^n\setminus \{0\}$;
\item [ \rm {(b)}]    there exist convex neighborhoods $U$ of $\ox$ and $V$ of $\ov$ and  constants $\kappa>0$ and $\ve>0$ such that 
$$
f(x')\ge f(x)+\la v,x'-x\ra +\frac{\kappa}{2}\|x-x'\|^2\quad \mbox{for all}\;\; (x,v)\in (\gph \sub f)\cap (U_\ve\times V), \; x'\in U,
$$
where $U_\ve=\{x\in U|\; f(x)< f(\ox)+\ve\}$.
 \end{itemize}
 \end{Proposition}
\begin{proof} If (b) holds, pick $w\in \R^n\setminus \{0\}$ and  arbitrary sequences $t_k\searrow0$, $(x^k,v^k)\to (\ox,\ov)$ with  $f(x^k)\to f(\ox)$ and  $(x^k,v^k)\in \gph \sub f$, and $w^k\to w$.
It is easy to see that $\Delta_{t_k}^2 f(x^k , v^k)(w^k) \ge {\kappa}\|w^k\|^2$ for any $k$ sufficiently large. 
Passing to the limit proves (a). To prove the opposite implication, we can argue by contradiction that if (b) fails, then for any $k\in \N$, we find $(x^k,v^k)\in \gph \sub f$   such  that $(x^k,v^k)\to (\ox,\ov)$ and $f(x^k)< f(\ox)+1/k$ and $y^k\in \R^n$ such that $y^k\to \ox$
for which we have 
$$
f(y^k)<f(x^k)+\la v^k,y^k-x^k\ra+\frac{1}{2k}\|x^k-y^k\|^2.
$$
Since $f$ is lsc, we also get $f(x^k)\to f(\ox)$. Setting $t_k=\|x^k-y^k\|$ and $w^k=(y^k-x^k)/t_k$, assume by passing to a subsequence if necessary that $w^k\to w$ for some $w\in \R^n\setminus \{0\}$.
Combining these shows that $\Delta_{t_k}^2 f(x^k , v^k)(w^k) <1/k$ for any $k$. Passing to the limit and taking into account that $w\neq 0$ lead us to a contradiction with (a) and hence complete the proof. 
\end{proof}

Note that $U_\ve$ in Proposition~\ref{uqgc}(b) can be replaced with $U$ provided that $f$ is subdifferentially continuous at $\ox$ for $\ov$. 
It is also important to mention that the property in Proposition~\ref{uqgc}(a) can be expressed equivalently as the existence of a constant $\ell>0$ for which  $\d^2_s f(\ox,\ov)(w)\ge \ell \|w\|^2$ 
for any $w\in \R^n$. Looking into the proof of  Proposition~\ref{uqgc}, one can conclude that if (b) holds with constant $\kappa>0$, then we have $\d^2_s f(\ox,\ov)(w)\ge \kappa \|w\|^2$ for any $w\in \R^n$.
Conversely, if $\d^2_s f(\ox,\ov)(w)\ge \ell \|w\|^2$ holds, we can see that (b) is satisfied for any constant $\kappa\in (0,\ell)$.

 Our first goal in this section is to establish a chain rule for the strict second subderivative of the function $\ph$ in \eqref{CF}. This will tell us when strict twice epi-differentiability  should be expected for such a composite  function.
We begin with the following result in  which a chain rule for the second subderivative was obtained  for the composite form in \eqref{CF} under the SOQC \eqref{soqc}.

\begin{Proposition} \label{ssub}Assume that $\ph:\R^n\to \oR$ has the representation \eqref{CF} around $\ox\in \R^n$, $\ov\in \sub \ph(\ox)$, and that $\olm\in \Lm(\ox,\ov)$. If the SOQC \eqref{soqc} holds at $(\ox,\olm)$, then there exists 
a neighborhood $U$ of $(\ox,\ov)$ such that for any $(x,v)\in U\cap \gph \sub \ph$, we have 
$$
\d^2 \ph(x,v)(w)= \la \lm,\nabla^2 \Phi(x)(w,w)\ra+ \d^2g(\Phi(x),\lm)(\nabla \Phi(x)w),
$$
where $\lm$ is the unique element of the Lagrange multiplier set $\Lm(x,v)$ and where  $\d^2g(\Phi(x),\lm)=\dd_{K_g(\Phi(x),\,\lm)}$.
\end{Proposition} 
\begin{proof} By Proposition~\ref{soqcg}, we find  $\dd>0$ such that   for all $(x,v)\in \big(\gph \sub \ph\big)\cap \B_{\dd}(\ox,\ov)$ the Lagrange multiplier set 
  $ \Lm(x,v)$ is a singleton and that for all $x\in \B_\dd(\ox)$ the basic constraint qualification in   \eqref{bcq} holds. 
 Combining these and \cite[Theorem~13.14]{rw} immediately gives us  the claimed formula for the second subderivative of $\ph$ for any $(x,v)\in \big(\gph \sub \ph\big)\cap \B_{\dd}(\ox,\ov)$.
 The formula for the second subderivative of $g$ is taken from \cite[Proposition~13.9]{rw}. 
\end{proof}

\begin{Theorem}[chain rule for strict second subderivative]\label{chss}
Assume that $\ph:\R^n\to \oR$ has the representation \eqref{CF} around $\ox\in \R^n$, $\ov\in \sub \ph(\ox)$, and that $\olm\in \Lm(\ox,\ov)$. If the SOQC \eqref{soqc} holds at $(\ox,\olm)$, then 
we have 
$$
\d_s^2 \ph(\ox,\ov)(w)= \la \olm,\nabla^2 \Phi(\ox)(w,w)\ra+\d^2_s g(\Phi(\ox),\olm)(\nabla \Phi(\ox)w)
$$
with $\d^2_s g(\Phi(\ox),\olm)(u)=\dd_{\spann\{K_ g(\Phi(\ox),\olm)\}}(u)$ for all $u\in \R^m$.
\end{Theorem}
\begin{proof} The given formula for the strict second subderivative of the polyhedral function $g$ was recently obtained in \cite[Proposition~4.1]{hjs}.
To justify the claimed formula for the strict second subderivative of $\ph$, take the neighborhood $O$ from Corollary~\ref{redu2} on which $\ph$ can be represented as  $\ph=\vth\circ \Psi$ with 
$\nabla \Psi(\ox)$ having full rank. Similar to \eqref{laset}, define the set of Lagrange multipliers at $(x,v)\in \gph \sub \ph$ associated with the latter representation of $\ph$ by 
$$
\Gamma(x,v)=\big\{\mu\in \R^s|\; \nabla \Psi(x)^*\mu=v, \; \mu \in \sub \vth(\Psi(x))\big\},
$$
where $s$ is taken from Corollary~\ref{redu2}. Note that the corresponding SOQC for the composite function $\vth\circ \Psi$ clearly holds since $\nabla \Psi(\ox)$ has full rank. 
When $(x,v)\in \gph \sub \ph$   is sufficiently close to $(\ox,\ov)$, it follows from Proposition~\ref{soqcg}(b) that $\Gamma(x,v)$ is a singleton 
and that $\|\mu-\bar\mu\|=O(\|x-\ox\|+\|v-\ov\|)$ with $\mu \in \Gamma(x,v)$ and $\bar\mu\in \Gamma(\ox,\ov)$. Pick $w\in \R^n$ and then take  arbitrary sequences $t_k\searrow0$, $(x^k,v^k)\to (\ox,\ov)$ with  $(x^k,v^k)\in \gph \sub \ph$, and $w^k\to w$ such that
$\Delta^2_{t_k}\ph(x^k,v^k)(w^k) \to \d_s^2\ph(\ox,\ov)(w)$ as $k\to\infty$.
Let $\mu^k\in \Gamma(x^k,v^k)$, set $u^{k}:=(\Psi(x^k+t_kw^k)-\Psi(\xk))/t_k$, and observe that  
\begin{align*}
\Delta^2_{t_k}\ph(x^k,v^k)(w^k)=&\;\dfrac{\vth\big(\Psi(x^k+t_kw^k)\big)-\vth\big(\Psi(x^k)\big)-t_k\langle\nabla\Psi(x^k)^*\mu^k,w^k\rangle}{\frac{1}{2}t_k^2}\\
=&\; \dfrac{\vth\big(\Psi(x^k)+t_ku^k\big)-\vth\big(\Psi(x^k)\big)-t_k\langle\mu^k,u^k\rangle}{\frac{1}{2}t_k^2}\\
&+ \dfrac{\langle\mu^k,\Psi(x^k+t_kw^k)-\Psi(\xk)-t_k\nabla\Psi(x^k)w^k\rangle}{\frac{1}{2}t_k^2}\\
=&\;\Delta^2_{t_k} \vth(\Psi(x^k),\mu^k)(u^k)+ \la \mu^k,\nabla^2\Psi(x^k)(w^k,w^k)\ra+ \dfrac{o(t^2_k)}{t_k^2}.
\end{align*}
Passing to the limit and taking into account that $\mu^k\to \bar\mu$ due to the discussion above  bring us to 
\begin{align}\label{ineq1}
\d_s^2\ph(\ox,\ov)(w)&\ge   \langle \bar\mu, \nabla^2\Psi(\ox) ( w,w)\rangle+\d_s^2\vth(\Psi(\ox),\bar\mu) (\nabla\Psi(\ox)w )\\
&= \langle \bar\mu, \nabla^2\Psi(\ox) ( w,w)\rangle+\dd_{\spann\{K_ \vth(\Psi(\ox),\bar\mu)\}} (\nabla\Psi(\ox)w )\nonumber,
\end{align}
where the last equality stems from the fact that $\vth$ is a polyhedral function. Clearly, we get equality if $\nabla\Psi(\ox)w \notin \spann\{K_ \vth(\Psi(\ox),\bar\mu)\}$.
We proceed to justify that  the opposite inequality in \eqref{ineq1} for any $w\in \R^n$ such that $\nabla\Psi(\ox)w \in \spann\{K_ \vth(\Psi(\ox),\bar\mu)\}$.
To this end,  it follows from \cite[Proposition~3.3]{hjs} that there exists a sequence  $\{(y^k,\mu^k)\b\subset \gph \sub \vth$ such that $(y^k,\mu^k)\to (\Psi(\ox),\bar\mu)$  
and that $\spann\{K_ \vth(\Psi(\ox),\bar\mu)\}=K_{\vth}(y^k,\mu^k)$ for any $k$.
In particular, we have $\nabla\Psi(\ox) w \in K_{\vth}(y^k,\mu^k)$ for all $k$. Since $\nabla \Psi(\ox)$ has full rank, we conclude from the classical Lyusternik-Graves Theorem (cf. \cite[Corollary~3.8]{m06})
that $\Psi$ is metrically regular  at $\ox$ for $\Psi(\ox)$. This implies that there exists a constant $\ell\ge 0$ such that the estimate 
$$
\dist (\ox, \Psi^{-1}(y^k))\le \ell\, \|\Psi(\ox)-y^k\|
$$
holds for all $k$ sufficiently large. Thus, we find $x^k\in \R^n$ such that $\Psi(x^k)=y^k$ and $x^k\to \ox$. The former yields $K_{\vth}(y^k,\mu^k)=K_{\vth}(\Psi(x^k),\mu^k)$. We can also 
use Proposition~\ref{mrcon} to conclude that the mappings $G_k(u):=\nabla \Psi(x^k)u-K_{\vth}(\Psi(x^k),\mu^k)$ are globally metrically regular  with a uniform constant $\kappa> 0$.
This brings us to the estimate 
\begin{align*}
\dist(w,G^{-1}_k(0))&\le \kappa\, \dist(0,G_k(w))=\kappa\, \dist\big(\nabla \Psi(x^k)w,K_{\vth}(\Psi(x^k),\mu^k)\big)\\
&\le \kappa\,\|\nabla \Psi(x^k)w-\nabla \Psi(\ox)w\| \le \kappa\,\|\nabla \Psi(x^k)-\nabla \Psi(\ox)\|\|w\|,
\end{align*}
where the penultimate step results from the fact that $\nabla\Psi(\ox)w \in K_{\vth}(\Psi(x^k), \mu^k)$.
This allows us to find $w^k\in G^{-1}_k(0)$ such that $ w^k\to w$. So, we have $\nabla \Psi(x^k)w^k\in K_{\vth}(\Psi(x^k),\mu^k)$ for all $k$ sufficiently large. 
Using Lemma~\ref{prof} and the given formula for the second subderivative of $\ph$ at $x^k$ for $v^k$ in Proposition~\ref{ssub}, we arrive at
\begin{align*}
\d_s^2\ph(\ox,\ov)(w)&\le \liminf_{k\to \infty}\d^2\ph(x^k,v^k)(w^k)\\
&=  \liminf_{k\to \infty}\langle  \mu^k, \nabla^2\Psi(x^k) ( w^k,w^k)\rangle+\d^2\vth(\Psi(x^k), \mu^k) (\nabla\Psi(x^k)w^k)\\
&=  \liminf_{k\to \infty} \langle  \mu^k, \nabla^2\Psi(x^k) ( w^k,w^k)\rangle+\dd_{K_\vth(\Psi(x^k),\, \mu^k)} (\nabla\Psi(x^k)w^k)\\
&=  \liminf_{k\to \infty} \langle  \mu^k, \nabla^2\Psi(x^k) ( w^k,w^k)\rangle= \langle \bar\mu, \nabla^2\Psi(\ox) ( w,w)\rangle\\
 &=\langle \bar\mu, \nabla^2\Psi(\ox) ( w,w)\rangle+\delta_{\spann\{K_ \vth(\Psi(\ox),\,\bar\mu)\}}(\nabla\Psi(\ox)w ).
\end{align*}
Combining this estimate and \eqref{ineq1} gives us 
\begin{equation}\label{chr1}
\d_s^2\ph(\ox,\ov)(w)=\langle \bar\mu, \nabla^2\Psi(\ox) ( w,w)\rangle+\d_s^2\vth(\Psi(\ox),\bar\mu) (\nabla\Psi(\ox)w ).
\end{equation}
 { To achieve the same result for the representation \eqref{CF}, recall first from Corollary~\ref{redu2} that $\Psi=B\circ \Phi$ and from Proposition~\ref{redu}  that  
$g$ can be  represented locally around $\Phi(\ox)$ as $g=\vth\circ B$. By the chain rule   from \cite[Theorem~10.6]{rw} and $B$ having full rank, we arrive at $\sub g(\Phi(\ox))=B^*\sub \vth(\Psi(\ox))$.
Thus, it follows  from $\nabla \Psi(\ox)=B\nabla \Phi(\ox)$ and $\bar\mu\in \Gamma(\ox,\ov)$     that $\olm:=B^*\bar\mu\in \Lm(\ox,\ov)$.  Considering the mapping 
$x\mapsto \la \olm, \Phi(x)\ra=  \la \bar \mu, \Psi(x)\ra$, we get   
\begin{equation}\label{chr1.1}
\la \bar\mu, \nabla^2 \Psi(\ox)(w,w)\ra=\la \olm, \nabla^2 \Phi(\ox)(w,w)\ra\quad \textrm{ for any }\quad w\in \R^n.
\end{equation}  }
 Since $B$ has full rank, we can apply \eqref{chr1} to the   representation $g=\vth\circ B$ and conclude that 
$$
\d_s^2g(\Phi(\ox),\olm) (\nabla\Phi(\ox)w )= \d_s^2\vth(\Psi(\ox),\bar\mu) (\nabla\Psi(\ox)w ).
$$
This, together with \eqref{chr1} and \eqref{chr1.1}, implies that 
\begin{equation*}
\d_s^2\ph(\ox,\ov)(w) = \langle  \olm, \nabla^2\Phi(\ox) ( w,w)\rangle+ \d_s^2g(\Phi(\ox),\olm) (\nabla\Phi(\ox)w ),
\end{equation*}
which proves the claimed formula for the strict second subderivative of $\ph$ and hence completes the proof.
\end{proof}

As an immediate consequence of the established chain rule for the strict second subderivative, we can achieve a useful characterization of a property 
called the   variational strong convexity.

\begin{Corollary}\label{vsc} Assume that $\ph:\R^n\to \oR$ has the representation \eqref{CF} around $\ox\in \R^n$, $\ov\in \sub \ph(\ox)$, and that $\olm\in \Lm(\ox,\ov)$. If the SOQC \eqref{soqc} holds at $(\ox,\olm)$, then
the following properties are equivalent:
\begin{itemize}[noitemsep,topsep=2pt]
\item [ \rm {(a)}]  $ \la \olm,\nabla^2 \Phi(\ox)(w,w)\ra >0$ for all $w\in \R^n\setminus\{0\}$ with  $\nabla\Phi(\ox)w\in \spann\{K_ g(\Phi(\ox),\olm)\}$;
\item [ \rm {(b)}]     $\ph$ is  variationally strongly convex at $\ox$ for $\ov$, meaning that there exist convex neighborhoods $U$ of $\ox$ and $V$ of $\ov$ and  an lsc convex function $h$ such that 
$(\gph \sub \ph)\cap (U\times V)=(\gph \sub h)\cap (U\times V)$ and that $h(x)\le \ph(x)$ for any $x\in U$ but $\ph(x)=h(x)$ for any $(x,v)\in (\gph \sub \ph)\cap (U\times V)$.
 \end{itemize}
\end{Corollary}
\begin{proof} According to \cite[Theorem~2]{r17}, if $\ov\in \rs \ph(\ox)$,  the variational strong convexity of  $\ph$ at $\ox$ for $\ov$ is equivalent to the uniform quadratic growth condition in Proposition~\ref{uqgc}(b)
with $f=\ph$.  The latter amounts to $\d_s^2\ph(\ox,\ov)(w)>0$ for all $w\in \R^n\setminus\{0\}$ due to  Proposition~\ref{uqgc}. Using the established chain rule in Theorem~\ref{chss} 
and taking into account that $\rs\ph (\ox)=\sub \ph(\ox)$ because of the SOQC  confirm that (a) and (b) are equivalent. 
\end{proof}

\begin{Remark} \label{tilts} One can use Corollary~\ref{vsc} to find a characterization of tilt-stable local minimizers of the composite optimization problem 
\begin{equation}\label{cp}
\mini \varphi(x) \quad\textrm{subject to} \; \; x\in \R^n.
\end{equation}
To elaborate more, consider a function $\ph$ with the representation \eqref{CF} around $\ox$. Recall that $\ox$ is said to be a tilt-stable local minimizer of $\ph$
if there exist neighborhoods $U$ of $\ox$ and $V$ of $\ov=0$ such that  the mapping 
$$
v\mapsto \argmin_{x\in U}\{\ph(x)-\la v,x-\ox\ra\}
$$
is single-valued and Lipschitz continuous on $V$ and its value at $\ov$ is $\{\ox\}$. Combining  \cite[Theorem~2]{r17} and \cite[Theorem~1.3]{pr1}   tells us that 
$\ox$ is   a tilt-stable local minimizer of $\ph$ if and only if $\ph$ is variationally strongly convex at $\ox$ for $\ov$. If the SOQC \eqref{soqc} holds at $(\ox,\olm)$, where $\olm$ is the unique element in $ \Lm(\ox,\ov)$, we conclude 
from Corollary~\ref{vsc} that   $\ox$ is   a tilt-stable local minimizer of $\ph$ if and only if we have 
$$
 \la \olm,\nabla^2 \Phi(\ox)(w,w)\ra >0\quad \mbox{ for all}\;\;  w\in \R^n\setminus\{0\}\;\;\mbox{with}\;\nabla\Phi(\ox)w\in \spann\{K_ g(\Phi(\ox),\olm)\},
 $$ a condition which was 
traditionally called   the {\em strong second-order sufficient condition}.  This characterization of tilt-stable local minimizers of \eqref{cp} was previously observed in \cite{ms16} using a different approach, relying mainly on the concept of 
coderivative. 
\end{Remark}

Next, we recall a characterization of strict twice epi-differentiability, which was established in \cite[Corollary~4.3]{pr2}. 
\begin{Proposition}[characterization of strict twice epi-differentiability] \label{sted}
Assume that $\ph$ has the representation \eqref{CF} around $\ox\in \R^n$ and    $ \ov\in  \sub \ph(\ox)$ and that the basic constraint qualification \eqref{bcq1} holds at $\ox$.
Then there is a neighborhood  $U$ of $(\ox,\ov)$ such that for any $(x,v)\in U\cap\, \gph \ph$, the following properties are equivalent:
\begin{itemize}[noitemsep,topsep=2pt]
\item [ \rm {(a)}] $\ph$ is strictly twice epi-differentiable at $x$ for $v$;
\item [ \rm {(b)}] $\d^2 \ph(x,v)$ epi-converges \rm{(}to something\rm{)}  as $(x',v')\to (x,v)$  with $v'\in \sub \ph(x')$.
\end{itemize}
\end{Proposition} 
\begin{proof} 
By  Proposition~\ref{phpro}(a), $\ph$ is  prox-regular and subdifferentially continuous
at  $\ox$ for $\ov$. Moreover, it follows from \cite[Theorem~13.14]{rw}, $\ph$ is twice epi-differentiable at $x$ for $v\in \sub \ph(x)$ whenever $(x,v)$ is so close to $(\ox,\ov)$ that \eqref{bcq} also holds at $x$.
The claimed equivalence is a special case  of a more general result in  \cite[Corollary~4.3]{pr2}, obtained  for prox-regular functions.
\end{proof}

We are now in a position to characterize  strict twice epi-differentiability of  the composite form \eqref{CF}. Our main tools will be a chain and  a sum rules for epi-convergence of a sequence of extended-real-valued functions,
established in \cite[Excercise~7.47]{rw} and \cite[Theorem~7.64]{rw}, respectively. Recall from \cite[page~250]{rw} that a sequence of extended-real-valued functions $\{f^k\b$ converges continuously to $f:\R^n\to \oR$
if for every $x\in \R^n$ and every sequence $x^k\to x$, we have $f^k(x^k)\to f(x)$. 

\begin{Theorem}\label{stepc}
Assume that $\ph$ has the representation \eqref{CF} around $\ox\in \R^n$, $\ov\in   \sub \ph(\ox)$, and that $\olm\in \Lm(\ox,\ov)$. If the SOQC \eqref{soqc} holds at $(\ox,\olm)$, then the following   properties are equivalent:
\begin{itemize}[noitemsep,topsep=2pt]
\item [ \rm {(a)}]   there exists a neighborhood $U$ of $(\ox,\ov)$ such that for any $(x,v)\in U\cap \gph \sub \ph$, $\ph$ is strictly twice epi-differentiable at $x$ for $v$;
\item [ \rm {(b)}]  $\ov\in \ri \sub \ph(\ox)$.
\end{itemize}
\end{Theorem}
\begin{proof} Suppose that (b) holds and take the neighborhood $U$ for which the conclusions in  Propositions~\ref{soqcg}(b) and \ref{ssub} hold.  By Proposition~\ref{phpro}(d),  we get
  $\olm\in \ri \sub g(\Phi(\ox))$. Using  \eqref{crieq} which holds under the condition  $\olm\in \ri \sub g(\Phi(\ox))$,   
we get $K_g(\Phi(\ox),\olm)=K_g(z,\lm)$  for all $(z,\lm)\in V\cap \gph \sub g$, where $V$ is taken from \eqref{crieq}. Shrinking the neighborhood $U$    if necessary and using \eqref{mslag}, we can assume that 
$(\Phi(x),\lm)\in V\cap \gph \sub g$, where $(x,v)\in U\cap \gph \sub \ph$ and $\lm$ is the unique Lagrange multiplier in $\Lm(x,v)$, which results from Proposition~\ref{soqcg}(b). Thus,   we arrive at 
\begin{equation}\label{cr21}
K_g(\Phi(\ox),\olm)=K_g(\Phi(x),\lm)\quad \mbox{for all}\;\; (x,v)\in U\cap \gph \sub \ph, \;\lm\in \Lm(x,v).
\end{equation} 
We claim that for any sequence $\{(x^k,v^k)\b\subset U\cap \gph \sub \ph$ such that $(x^k,v^k)\to (\ox,\ov)$, we have 
\begin{equation}\label{ep}
\d^2 \ph(x^k,v^k)\xrightarrow{e} \d^2 \ph(\ox,\ov).
\end{equation}
To this end, pick $w\in \R^n$ and conclude from  Proposition~\ref{ssub} that 
$$
\d^2 \ph(x^k,v^k)(w)=f^k(w)+\dd_{K_g(\Phi(x^k),\,\lm^k)}(\nabla\Phi(x^k)w)\quad \mbox{with}\;\; f^k(w):= \la \lm^k,\nabla^2 \Phi(x^k)(w,w)\ra,
$$
where $\lmk$ is the unique element in $\Lm(\xk, v^k)$. Since $(x^k,v^k)\to (\ox,\ov)$, we deduce from \eqref{mslag} that $\lm^k\to \olm$. This tells us that the sequence of functions $\{f^k\b$ is continuously convergent to $f$, where 
$f(w)=\la \olm,\nabla^2 \Phi(\ox)(w,w)\ra$ for any $w\in \R^n$. Moreover, it follows from \eqref{cr21} that $K_g(\Phi(\ox),\olm)=K_g(\Phi(x^k),\lm^k)$, which immediately implies   that 
$\dd_{K_g(\Phi(x^k),\,\lm^k)}\xrightarrow{e} \dd_{K_g(\Phi(\ox),\,\olm)}$. Appealing now to \cite[Exercise~7.47(a)]{rw} tells us that 
\begin{equation}\label{cq4}
\dd_{K_g(\Phi(x^k),\,\lm^k)}\circ\nabla\Phi(x^k) \xrightarrow{e} \dd_{K_g(\Phi(\ox),\,\olm)} \circ \nabla\Phi(\ox)
\end{equation}
provided that the condition 
\begin{equation}\label{cq3}
0\in \inte \big(K_g(\Phi(\ox),\olm)-\rge \nabla \Phi(\ox)\big)
\end{equation}
is satisfied. According to \cite[Proposition~2.97]{bs}, this condition amounts to
$$
K_g(\Phi(\ox),\olm)^*  \cap  \ker \nabla \Phi(\ox)^*=\{0\}.
$$
By \eqref{cal2d}, we have $K_g(\Phi(\ox),\olm)^*\subset D^*(\sub g)\big(\Phi(\ox),\olm\big)(0)$, which together  with the SOQC \eqref{soqc} confirms that  \eqref{cq3} is satisfied. 
This proves \eqref{cq4}. Appealing now to the sum rule for epi-convergence from  \cite[Theorem~7.46(b)]{rw} justifies  \eqref{ep}. Thus, it results from Proposition~\ref{sted} that $\ph$ is strictly twice epi-differentiable at
$\ox$ for $\ov$. To achieve the same conclusion for any point $(x,v)\in \gph \sub \ph$ sufficiently close to $(\ox,\ov)$, we claim that for any such a pair $(x,v)$, we have $v\in \ri \sub \ph(x)$. As pointed out above,  
  the condition $\ov\in \ri \sub \ph(\ox)$ yields   $\olm\in \ri \sub g(\Phi(\ox))$. Since $K_g(\Phi(\ox),\olm)=N_{\sub g(\Phi(\ox))}(\olm)$, we conclude from Proposition~\ref{phpro}(c) that the critical cone $K_g(\Phi(\ox),\olm)$ is a linear subspace.
  This, combined with \eqref{cr21}, ensures that $K_g(\Phi(x),\lm)$ is a linear subspace for any  $(x,v)\in U\cap \gph \sub \ph$ and the unique Lagrange multiplier $\lm$ in $\Lm(x,v)$. Again, since 
   $K_g(\Phi(x),\lm)=N_{\sub g(\Phi(x))}(\lm)$, we arrive at $\lm\in \ri \sub g(\Phi(x))$. It follows from Proposition~\ref{phpro}(d) that $v\in \ri \sub \ph(x)$, which proves our claim. Employing the same argument as the one 
   for $(\ox,\ov)$ indicates that $\ph$ is strictly twice epi-differentiable at $x$ for $v$  for any  $(x,v)\in U\cap \gph \sub \ph$ and hence proves (a).
   
   Assume now that (a) holds. In particular, $\ph$ is strictly twice epi-differentiable at $\ox$ for $\ov$. This implies that $ \d_s^2 \ph(\ox,\ov)= \d^2 \ph(\ox,\ov)$.
   According to \cite[Theorem~13.14]{rw}, we have $\dom  \d^2 \ph(\ox,\ov)=K_{\ph}(\ox,\ov)=N_{\sub \ph(\ox)}(\ov)$. We claim that $K_{\ph}(\ox,\ov)$ is a linear subspace.
   To justify it, it suffices to show that if $w\in K_{\ph}(\ox,\ov)$, we have $-w\in K_{\ph}(\ox,\ov)$.
   Take $w\in K_{\ph}(\ox,\ov)$. By Proposition~\ref{phpro}(b),  we get  $\nabla \Phi(\ox)w\in K_g(\Phi(\ox),\olm) \subset \spann\{K_g(\Phi(\ox),\olm)\}$, which implies that  $-\nabla \Phi(\ox)w\in \spann\{K_g(\Phi(\ox),\olm)\}$.
This, combined with Theorem~\ref{chss}, confirms that $ \d_s^2 \ph(\ox,\ov)(-w)<\infty$ and hence  $ \d^2 \ph(\ox,\ov)(-w)<\infty$, which implies that  $-w\in K_{\ph}(\ox,\ov)$.
Thus,  $N_{\sub \ph(\ox)}(\ov)$ is a linear subspace and hence  $\ov\in \ri \sub \ph(\ox)$, which completes the proof. 
\end{proof}

Note that when the polyhedral function $g$ in Theorem~\ref{stepc}  is the pointwise maximum 
$g(z)=\max\{z_1,\ldots,z_m\}$ with $z=(z_1,\ldots,z_m)$ as in Example~\ref{point}(a), Theorem~\ref{stepc} covers  \cite[Proposition~3.8]{pr3} with a different proof.
As pointed out in Example~\ref{point}(a),  in this case, the SOQC \eqref{soqc}  is equivalent  to saying that the vectors $\nabla\ph_1(\ox),\ldots,\nabla\ph_m(\ox)$ are affinely independent. 
 When $g$ is given as the one in Example~\ref{point}(b), Theorem~\ref{stepc}  reduces to  \cite[Proposition~4.13]{pr3}. We should add here both results in \cite{pr3} only 
 contain the implication (b)$\implies $(a) in Theorem~\ref{stepc} and do not achieve the established equivalence of  (a) and (b) as in Theorem~\ref{stepc}.  
 % Note also that the SOQC in \eqref{soqc}, utilized in Theorem~\ref{stepc}, is not essential to ensure strict twice epi-differentiability of  the composite function in \eqref{CF}.
%Indeed, it was shown  in \cite[Example~4.11]{pr3} that the function $\ph(x_1,x_2)=|x_1^3x_2^2|$ is strictly twice epi-differentiable at $(0,0)$ for $(0,0)$.
%This function can be written in the form of \eqref{CF} by setting $g(z_1,z_2)=\max\{z_1,z_2\}$ and $\Phi(x_1,x_2)=(x_1^3x_2^2,-x_1^3x_2^2)$. According 

 We continue by exploring an important geometric property of $\gph \sub \ph$, called strict proto-differentiability, which will be of great importance in the next section. To this end, take 
  a set-valued mapping $F:\R^n\tto \R^m$ and   recall from \cite[page~331]{rw} that   $F$ is said to be {\em proto-differentiable} at $\ox$ for $\oy\in F(\ox)$ if the outer graphical limit in \eqref{proto} is actually a full limit.
  When $F(\ox)$ is a singleton consisting of $\oy$ only, the notation $DF(\ox, \oy)$ is simplified to $DF(\ox)$. It is easy to see that 
for a single-valued mapping $F$, which is differentiable at $\ox$,  the graphical derivative  $DF(\ox)$ boils down to the Jacobian matrix of $F$ at $\ox$, denoted by $\nabla F(\ox)$. 
Recall from \cite[Definition~9.53]{rw} that the strict graphical derivative 
of  a set-valued mapping $F$ at $\ox$ for $\oy$, is  the set-valued mapping $\widetilde D F(\ox,\oy):\R^n\tto \R^m$, defined  by 
\begin{equation}\label{sproto}
\gph \widetilde DF(\ox,\oy)=\limsup_{
  \substack{
   t\searrow 0 \\
  (x,y)\xrightarrow{ \gph F}(\ox,\oy)
  }} \frac{\gph F-(x,y)}{t}.
\end{equation}
The set-valued mapping $F$ is said to be {\em strictly} proto-differentiable at $\ox$ for $\oy$ if the outer graphical limit in \eqref{sproto} is   a full limit.
 To shed more light on this definition,  take a set  $\Omega\subset \R^d$ and  $\ox\in \Omega$. The regular (Clarke) tangent cone and the paratingent cone to   $\Omega$ at $\ox$ are defined, respectively,   by
$$
 \rt_{\Omega}(\ox) =\liminf_{x \xrightarrow{ \Omega}\ox,t\searrow 0} \frac{\Omega-x}{t}\quad \mbox{and}\quad  \widetilde {T}_{\Omega}(\ox) =\limsup_{x \xrightarrow{ \Omega}\ox,t\searrow 0} \frac{\Omega-x}{t}.
$$
Since we always have $ \rt_{\Omega}(\ox)\subset  T_{\Omega}(\ox)\subset  \widetilde {T}_{\Omega}(\ox)$, strict proto-differentiability of a set-valued mapping $F$ at $\ox$ for $\oy$ amounts to the condition 
$ \rt_{\gph F}(\ox,\oy)= \widetilde {T}_{\gph F}(\ox,\oy)$. Note that if  the latter equality  holds, we must have $\widetilde D F(\ox, \oy)=DF(\ox, \oy)$. We recently showed in \cite[Theorem~3.5(c)]{hjs}
that if $g:\R^m\to \oR$ is a polyhedral function and $(\oz,\olm)\in \gph \sub g$, then $\sub g$ is strictly proto-differentiable at $\oz$ for $\olm$ if and only if $\olm\in \ri \sub g(\oz)$.  Below, we are going to demonstrate that a similar result can be expected for the composite form in \eqref{CF} when the SOQC \eqref{soqc} is satisfied.
 
\begin{Theorem}\label{prot}
Assume that $\ph$ has the representation \eqref{CF} around $\ox\in \R^n$, $\ov\in   \sub \ph(\ox)$, and that $\olm\in \Lm(\ox,\ov)$. If the SOQC \eqref{soqc} holds at $(\ox,\olm)$, then the following   properties are equivalent:
\begin{itemize}[noitemsep,topsep=2pt]
\item [ \rm {(a)}]   there exists a neighborhood $U$ of $(\ox,\ov)$ such that for any $(x,v)\in U\cap \gph \sub \ph$, $\sub \ph$ is strictly proto-differentiable at $x$ for $v$;
\item [ \rm {(b)}]  $\ov\in \ri \sub \ph(\ox)$.
\end{itemize}
\end{Theorem} 
 \begin{proof} As explained in the proof of Proposition~\ref{sted}, the composite function $\ph$ satisfies the assumptions of \cite[Corollary~4.3]{pr2}. The latter result demonstrates that 
 the property in (a) is equivalent to the property  in Theorem~\ref{stepc}(a). This, combined with  Theorem~\ref{stepc}, proves the equivalence of (a) and (b).  
 \end{proof}
 
 We continue by characterizing the regularity of the graph of the subgradient mappings of functions with the composite representation  \eqref{CF}. 
 Recall that a set $C\subset \R^d$ is {\em regular} at $\ox\in C$ provided that $\rN_C(\ox)=N_C(\ox)$.
 To achieve our goal, we begin with an observation about the regular tangent cone.  
 \begin{Lemma}\label{subsp} Assume that $\ph:\R^n\to \oR$ has the representation \eqref{CF} around $\ox\in \R^n$, $\ov\in \sub \ph(\ox)$, and the basic constraint 
 qualification \eqref{bcq1} holds at $\ox$. Then the regular tangent cone $ \rt_{\gph \sub \ph}(\ox,\ov)$ is a linear subspace. 
 \end{Lemma}
 \begin{proof} By Proposition~\ref{phpro}(a), the composite function $\ph$, accompanied with \eqref{bcq1}, is prox-regular and subdifferentially continuous 
 at $\ox$ for $\ov$. Thus, it results from \cite[Proposition~13.46]{rw} that the subgradient mapping $\sub \ph:\R^n\tto\R^n$ is  graphically  Lipschitzian   of dimension $n$ around $(\ox,\ov)$, meaning that
 $\gph \sub \ph$ can be identified locally with the graph of a Lipschitz continuous mapping from $\R^n$ into $\R^n$; see \cite[Definition~9.66]{rw} for more elaboration on this concept.
Employing now \cite[Theorem~3.5(b)]{r85} tells us that $ \rt_{\gph \sub \ph}(\ox,\ov)$ is a linear subspace.
 \end{proof}
 
  \begin{Theorem} \label{regch} Assume that $\ph:\R^n\to \oR$ has the representation \eqref{CF} around $\ox\in \R^n$, $\ov\in \sub \ph(\ox)$, and that $\olm\in \Lm(\ox,\ov)$. If the SOQC \eqref{soqc} holds at $(\ox,\olm)$, then
 the following properties are equivalent:
 \begin{itemize}[noitemsep,topsep=2pt]
 \item [ \rm {(a)}]     $\ov\in \ri \sub \ph(\ox) $;  
\item [ \rm {(b)}]  $\gph \sub \ph$ is regular at $(\ox,\ov)$, that is 
$
\rN_{\gph\sub \ph}(\ox,\ov)=N_{\gph\sub \ph}(\ox,\ov)
$;
\item [ \rm {(c)}]   for any $w\in \R^n$, we have $D(\sub \ph)(\ox,\ov)(w)=D^*(\sub \ph)(\ox,\ov)(w)$.
\end{itemize}
 \end{Theorem} 
 \begin{proof} Suppose that  $\ov\in \ri \sub \ph(\ox) $. It follows from  Theorem~\ref{prot} that $\sub \ph$ is strictly proto-differentiable at $\ox$ for $\ov$.
 This is equivalent to saying that $ \rt_{\gph \sub \ph}(\ox,\ov)=\widetilde{T}_{\gph \sub \ph}(\ox,\ov)$, which implies that $\rt_{\gph \sub \ph}(\ox,\ov)= T_{\gph \sub \ph}(\ox,\ov)$. 
 By \cite[Corollary~6.29]{rw}, the latter is equivalent to  saying that $\rN_{\gph\sub \ph}(\ox,\ov)=N_{\gph\sub \ph}(\ox,\ov)$, which proves (b). 
 Conversely, assume that (b) holds. As pointed out above, this implies that $\rt_{\gph \sub \ph}(\ox,\ov)= T_{\gph \sub \ph}(\ox,\ov)$.
 We claim that $K_{\ph}(\ox,\ov)$ is a linear subspace. 
 By Proposition~\ref{phpro}(b), we have    $K_{\ph}(\ox,\ov)=N_{\sub \ph(\ox)}(\ov)$. Thus, to verify our claim, it suffices to show that if $w\in K_{\ph}(\ox,\ov)$, then $-w\in K_{\ph}(\ox,\ov)$.
 Take $w\in K_{\ph}(\ox,\ov)$ and observe  from   \cite[Theorem~7.2]{mms0} that $\dom D(\sub \ph)(\ox,\ov)=K_{\ph}(\ox,\ov)$. So, we find $u\in \R^n$ such that $u\in D(\sub \ph)(\ox,\ov)(w)$ or equivalently 
 $(w,u)\in T_{\gph \sub \ph}(\ox,\ov)$. By Lemma~\ref{subsp} and $\rt_{\gph \sub \ph}(\ox,\ov)= T_{\gph \sub \ph}(\ox,\ov)$, we conclude that $(-w,-u)\in T_{\gph \sub \ph}(\ox,\ov)$. This can be translated as
 $-u\in D(\sub \ph)(\ox,\ov)(-w)$. Thus, we get that $-w\in \dom D(\sub \ph)(\ox,\ov)=K_{\ph}(\ox,\ov)$, which proves our claim. Since $N_{\sub \ph(\ox)}(\ov)$ is a linear subspace, we arrive at  $\ov\in \ri \sub \ph(\ox) $, and hence obtain (a). 
 
 Now,  assume (a) holds. This implies via Proposition~\ref{phpro}(d) that      $\olm\in \ri \sub g(\Phi(\ox))$, where $\olm$ is the unique Lagrange multiplier in $\Lm(\ox,\ov)$. Appealing to  \cite[Corollary~3.7]{hjs} indicates that 
 the condition $\olm\in \ri \sub g(\Phi(\ox))$ is equivalent to 
 \begin{equation}\label{riq}
 D(\sub g)(\Phi(\ox), \olm)=D^*(\sub g)(\Phi(\ox), \olm).
 \end{equation}
 Recall that under the SOQC \eqref{soqc} it holds that  $\Lm(\ox,\ov)=\{\olm\}$. Since the basic constraint qualification \eqref{bcq1} holds, it follows from \cite[Theorem~7.2]{mms0} that 
\begin{equation}\label{gdph}
 D(\sub \ph)(\ox,\ov)(w)=\nabla^2\la \olm,\Phi\ra(\ox)w+\nabla\Phi(\ox)^*D(\sub g)(\Phi(\ox), \olm)(\nabla\Phi(\ox)w).
\end{equation}
 Also, one can see from \cite[Theorem~4.3]{mr} that 
\begin{equation}\label{coph}
 D^*(\sub \ph)(\ox,\ov)(w)=\nabla^2\la \olm,\Phi\ra(\ox)w+\nabla\Phi(\ox)^*D^*(\sub g)(\Phi(\ox), \olm)(\nabla\Phi(\ox)w).
\end{equation}
 Combining these and \eqref{riq} proves that $D(\sub \ph)(\ox,\ov)=D^*(\sub \ph)(\ox,\ov)$. Conversely, suppose that  (c) holds.
 This, in particular, indicates that $\dom D(\sub \ph)(\ox,\ov)=\dom D^*(\sub \ph)(\ox,\ov) $. 
 By the second equation in \eqref{domgd} and \eqref{coph}, we get
 $$
 \dom D^*(\sub \ph)(\ox,\ov)=\big\{w\in \R^n|\; \nabla\Phi(\ox)w\in \spann\{K_g(\Phi(\ox), \olm)\}\big\},
 $$
 which clearly is a linear subspace. Thus,  $\dom D(\sub \ph)(\ox,\ov)=K_{\ph}(\ox,\ov)=N_{\sub \ph(\ox)}(\ov)$ is a linear subspace. By Proposition~\ref{phpro}(c), we get $\ov\in \ri \sub \ph(\ox)$ and hence obtain (a).
 \end{proof}
 
We close this section by considering a special case of the polyhedral function $g$ in the composite form \eqref{CF}. 
Assume that $\Th$ is a subset of $\R^n$ and $\ox\in \Th$ and that there exists a neighborhood $O$ of $\ox$ on which  $\Th$
has the   representation
\begin{equation}\label{CF2}
\Th\cap O=\big\{x\in O|\, \Phi(x)\in C\big\},
\end{equation}
where  $C$ is a polyhedral convex set in $\R^m$ and that $\Phi$ is taken from \eqref{CF}. This clearly fits into the composite form \eqref{CF}
by taking that $g=\dd_C$. An important example of the set $\Th$ with the representation \eqref{CF2} is when the polyhedral convex set $C=\{0\}$ in $\R^m$. In this case, it follows from 
Example~\ref{point}(b) that the SOQC \eqref{soqc} amounts   to  $\nabla \Phi(\ox)$ having full rank. Note that a set $\Th\subset \R^n$ is said to be a  {\em ${\cal C}^2$-smooth manifold} around a point $\ox \in \Th$ if there exists 
a neighborhood $O$ of $\ox$ on which $\Th$ has the representation \eqref{CF2} with $C=\{0\}\subset\R^m$ and $\nabla \Phi(\ox)$ has full rank.
 For such a set, it was shown recently in \cite[Theorem~6.3]{lw} via a different approach that   both properties  in Theorem~\ref{regch}(b)-(c)
 hold. It is important to mention that our results, Theorems~\ref{prot} and \ref{regch}, go one step further and  show that, indeed, the regular tangent cone and the paratingent cone to $\gph N_\Th$ coincide
when $\Th$ is a ${\cal C}^2$-smooth manifold.

 %%%%%%%%%%%%%%%%%%%%%%%%%%%%%%%%%%%%%%%%

 \section{Stability Properties  of Generalized Equations}\label{sec04}
 
This section concerns the stability properties of the solution set to the canonical perturbation of the generalized equation 
\begin{equation}\label{GE}
\ou\in f(x)+ \sub \ph(x),
\end{equation}
where $f:\R^n\to \R^n$ is a ${\cal C}^1$ function and   $\ph$ is taken from \eqref{CF} and where $\ou$ is a fixed vector in $\R^n$. 
This generalized equation provides a rich and unified framework to study stability properties of variational inequalities and 
KKT systems from important classes of constrained and composite optimization problems. In particular, when $\ph=\dd_C$
with $C$   a polyhedral convex set in $\R^n$, it reduces to the generalized equation in  \cite{dr96} by Dontchev and Rockafellar  for which they demonstrated  that strong metric regularity 
and metric regularity of \eqref{GE} are equivalent. 
Recall that a set-valued mapping $F:\R^n \tto \R^m$ is called {\em strongly metrically regular} at $\ox$ for $\oy\in F(\ox)$ if $F^{-1}$ admits a Lipschitz continuous single-valued localization around $\oy$ for $\ox$, which means that 
there exist neighborhoods $U$ of $\ox$ and $V$ of $\oy$ such that the mapping $y\mapsto F^{-1}(y)\cap U$ is single-valued and Lipschitz continuous on $V$. According to 
\cite[Proposition~3G.1]{DoR14}, strong metric regularity of $F$ at $\ox$ for $\oy$ amounts to 
  $F$ being metrically regular at $\ox$ for $\oy$ and its inverse $F^{-1}$ admitting a single-valued  localization around $\oy$ for $\ox$.
  
Our first goal is to show that for the mapping $G:\R^n\tto\R^n$, defined  by 
\begin{equation}\label{mapg}
G(x)= f(x)+ \sub \ph(x), \;\;\;x\in \R^n,
\end{equation}
 metric regularity  and strong metric regularity at $\ox$ for $\ou$, where $\ox$ is a solution of \eqref{GE} satisfying  the nondegeneracy condition \eqref{nds} below, are equivalent. 
Our approach is different from  \cite{dr96} and relies upon the characterization of  strict proto-differentiability of the subgradient mapping $\sub \ph$, established in Theorem~\ref{prot}.
This forces us to confine our analysis to nondegenerate solutions to \eqref{GE}. Recall that $\ox$ is called a {\em nondegenerate} solution to the generalized equation \eqref{GE} if it satisfies the relative interior 
condition 
\begin{equation}\label{nds}
\ou -f(\ox)\in \ri \sub \ph(\ox).
\end{equation}
The following result presents an equivalent description of nondegenerate solutions to \eqref{GE}.
\begin{Proposition}\label{chnon} Assume that $\ox$ is a solution to the generalized equation \eqref{GE} and that the basic constraint qualification in \eqref{bcq1} holds at $\ox$. Then the following properties are equivalent:
\begin{itemize}[noitemsep,topsep=2pt]
\item [ \rm {(a)}]  $\ox$ is a nondegenerate solution to \eqref{GE};
\item [ \rm {(b)}]    the critical cone $K_{\ph}(\ox,\ou-f(\ox))$ is a linear subspace.
\end{itemize}
\end{Proposition}
\begin{proof}  According to \eqref{criph}, we have $K_{\ph}(\ox,\ou -f(\ox))=N_{\sub \ph(\ox)}(\ou -f(\ox))$. Since $\sub \ph(\ox)$
is  convex, we can conclude that $K_{\ph}(\ox,\ou-f(\ox))$ is a linear subspace if and only if \eqref{nds} is satisfied. This proves the equivalence of (a) and (b).
\end{proof}

 To proceed, define  the solution mapping $S: \R^n \tto \R^n$ to the canonical perturbation of   \eqref{GE}  by
\begin{equation}\label{mr33}
S(u):= G^{-1}(u)=\big\{x\in \R^n\, \big|\, u\in f(x)+\partial \ph(x)\big\}, \quad u\in \R^n.
\end{equation}

 We begin with the following 
result in which metric regularity of $G$ will be characterized.
\begin{Theorem} \label{mrch}Assume that $\ox$ is a nondegenerate solution to \eqref{GE} and that the SOQC in \eqref{soqc} holds at $(\ox,\olm)$, where $\olm$ is the unique vector in $\Lm(\ox,\ou-f(\ox))$.
Set $\K:=K_\ph(\ox,\ou-f(\ox))$ and $A:=\nabla f(\ox)+\nabla^2\la \olm,\Phi\ra(\ox)$. Then
 the following properties are equivalent:
 \begin{itemize}[noitemsep,topsep=2pt]
\item [ \rm {(a)}]   the mapping $G$ in \eqref{mapg} is metrically regular at $\ox$ for $\ou$;
\item [ \rm {(b)}]  $\big\{w\in \R^n|\; A^*w\in \K^{\perp}\big\}\cap \K=\{0\}$;
\item [ \rm {(c)}]   $A\K+\K^\perp = \R^n$;
\item [ \rm {(d)}]  $DG(\ox, \ou) $ is surjective, meaning that $\rge DG(\ox, \ou)=\R^n$. 
\end{itemize}
\end{Theorem}
\begin{proof} By \eqref{nds} and Theorem~\ref{regch},  we get $D(\sub \ph)(\ox,\ou-f(\ox))=D^*(\sub \ph)(\ox,\ou-f(\ox))$.
This, along with \eqref{gdph}, implies for any $w\in \R^n$ that 
\begin{align}
D^*G(\ox,\ou)(w)&=\nabla f(\ox)^*w+D^*(\sub \ph)(\ox,\ou-f(\ox))(w)\nonumber\\
&= A^*w+ \nabla\Phi(\ox)^*D(\sub g)(\Phi(\ox), \olm)(\nabla\Phi(\ox)w)\nonumber\\
&= A^*w+ \nabla\Phi(\ox)^*N_{K_g(\Phi(\ox), \olm)}(\nabla\Phi(\ox)w)\nonumber\\
&= A^*w+N_{\K}(w)\label{calco},
\end{align}
where the penultimate step results from \eqref{cal2d} and the last equality comes from Proposition~\ref{phpro}(b).  By Proposition~\ref{chnon},   $\K$ is a linear subspace, which brings us to 
$$
q\in D^*G(\ox,\ou)(w)\iff q-A^*w\in \K^\perp, \;\; w\in \K.
$$
It is well-known (cf. \cite[Theorem~9.43]{rw}) that metric regularity of $G$ at $\ox$ for $\ou$ can be characterized by the implication 
$$
0\in D^*G(\ox,\ou)(w)\implies w=0.
$$
 Combining this with the calculation of $D^*G(\ox,\ou)(w)$, obtained above, proves the equivalence of (a) and (b). 
Employing again the fact that $\K$ is a linear subspace, we get by \cite[Corollary~11.25(c)]{rw} that
 $$
 (A\K)^\perp =\big\{w\in \R^n|\; A^*w\in \K^{\perp}\big\},
 $$
 which implies  (b) and (c) are equivalent due to the fact that $\K$ is a linear subspace. Arguing similar to \eqref{calco}, we get for any $w\in \R^n$ that 
\begin{equation}\label{calgd}
 DG(\ox,\ou)(w)=Aw+ N_{\K}(w),
\end{equation}
 which brings us to 
\begin{equation}\label{calgd}
 q\in DG(\ox,\ou)(w) \iff q\in Aw+\K^\perp, \;\; w\in \K.
\end{equation}
 This tells us that $\rge DG(\ox,\ou)= A\K+\K^\perp$. Thus,   $\rge DG(\ox, \ou)=\R^n$ if and only if $A\K+\K^\perp=\R^n$, which demonstrates that (c) and (d) are equivalent.
\end{proof}

Next, we are going to show that metric regularity and strong metric regularity of the mapping $G$ in \eqref{mapg} are equivalent for nondegenerate solutions to \eqref{GE}.
To achieve our goal, we are going to utilize the following characterization of strong metric regularity from 
\cite[Theorem~4D.1]{DoR14}:  a set-valued mapping   $F:\R^n\tto\R^m$   is strongly metrically regular at $\ox$ for $\oy\in F(\ox)$ provided that it has a locally closed graph and 
the conditions 
 \begin{equation}\label{mr9}
0\in \widetilde DF(\ox, \oy)(w)\; \Longrightarrow\; w=0
\end{equation}
 and 
\begin{equation}\label{mr7}
\ox \in \liminf_{y\to \oy} F^{-1}(y)
\end{equation}
hold.  

\begin{Theorem}[equivalence of metric regularity and strong metric regularity] \label{mrsmr}
Assume that $\ox$ is a nondegenerate solution to \eqref{GE} and that the SOQC in \eqref{soqc} holds at $(\ox,\olm)$, where $\olm$ is the unique vector in $\Lm(\ox,\ou-f(\ox))$.
Set $\K:=K_\ph(\ox,\ou-f(\ox))$ and $A:=\nabla f(\ox)+\nabla^2\la \olm,\Phi\ra(\ox)$. Then the following properties hold.
\begin{itemize}[noitemsep,topsep=2pt]
\item [ \rm {(a)}]   The mapping $G$ in \eqref{mapg} is metrically regular at $\ox$ for $\ou$ if and only if it is strongly metrically regular at $\ox$ for $\ou$.
\item [ \rm {(b)}]   One of the equivalent properties {\em(}a{\rm)}-{\rm(}d{\rm)} in Theorem~{\rm\ref{mrch}} holds if and only if the solution mapping $S$ in \eqref{mr33} has a Lipschitz continuous localization $\sigma$ around $\ou$ for $\ox$, which 
is  ${\cal C}^1$ in a neighborhood of $\ou$ and 
$$
\nabla \sigma(\ou)=B\big(B^*AB\big)^{-1}B^*,
$$
 where $B\in \R^{n\times s}$ is a matrix whose columns form a basis for the linear subspace $\K$ and $s=\dim \K$.
\end{itemize}
\end{Theorem}
\begin{proof} We begin with the proof of (a). If $G$ is strongly metrically regular at $\ox$ for $\ou$, it clearly enjoys metric regularity at $\ox$ for $\ou$. Suppose now that 
$G$ is metrically regular at $\ox$ for $\ou$. According to the discussion above, it suffices to verify the validity of \eqref{mr9} and \eqref{mr7}. The latter directly results from 
the estimate \eqref{mr8} for $F=G$. To prove the former, take $w\in \R^n$ such that $0\in \widetilde DG(\ox, \ou)(w)$. We know from the SOQC and Theorem~\ref{regch}
that $\sub \ph$ is strictly proto-differentiable at $\ox$ for $\ou-f(\ox)$ due to \eqref{nds}. This, coupled with \cite[Proposition~5.3]{hjs}, indicates that $G$ 
is strictly proto-differentiable at $\ox$ for $\ou$. The discussion prior to Theorem~\ref{prot} tells us that strict proto-differentiability of $G$ at $\ox$ for $\ou$ yields 
$\widetilde DG(\ox, \ou)=DG(\ox, \ou)$. Thus, by  \eqref{calgd}, the condition $0\in \widetilde DG(\ox, \ou)(w)$ is equivalent to saying that $-Aw\in \K^\perp$ and $w\in \K$. 
Since $\K$ is a linear subspace of dimension $s$, we find a matrix $B\in \R^{n\times s}$ whose columns form a basis for $\K$. This means that $\K=\rge B$,
which is equivalent to $\K^\perp=\ker B^*$. This allows us to deduce from the conditions $-Aw\in \K^\perp$ and $w\in \K$ that there exists $q\in \R^s$ such that $w=Bq$
and $B^*ABq=0$. We are going to show that the $s\times s$ matrix $B^*AB$ is nonsingular. To do so, pick $p\in \R^n$.  Since $G$ is   metrically regular at $\ox$ for $\ou$, 
it results from Theorem~\ref{mrch}(c) that there are $w'\in\K$ and $p'\in \K^\perp$ for which we have $p=Aw'+p'$. By $\K=\rge B$, we find $d\in \R^s$ that $w'=Bd$.
Combining these and using $\K^\perp=\ker B^*$, we arrive at 
$$
B^*p=B^*(Aw'+p')=B^*Aw'=B^*ABd.
$$
Because  $p$ was chosen arbitrarily, we get $\rge B^*\subset \rge (B^*AB)$. Since the opposite inclusion always holds, we obtain $\rge B^*= \rge (B^*AB)$.
This, together with the fact that $B$ has full column rank, implies that $ \rge (B^*AB)=\R^s$ and thus confirms that $B^*AB$ is nonsingular. Consequently, 
the condition  $B^*ABq=0$ results in $q=0$ and so $w=Bq=0$. This proves \eqref{mr9} and completes the proof of (a).

Turning to the proof (b), remember first that $S=G^{-1}$. So, the claimed equivalence in (b) falls directly out of Theorem~\ref{mrch} and (a).
Next, we are going to show that if one of the equivalent properties { (}a{)}-{(}d{)} in Theorem~{\rm\ref{mrch}} holds, the 
 Lipschitz continuous localization of $S$ around $\ou$ for $\ox$ is ${\cal C}^1$ around $\ou$. To this end,  we find neighborhoods $U$ of $\ou$ and $V$
of $\ox$ such that the mapping $u\mapsto S(u)\cap V$ is single-valued and Lipschitz continuous on $U$. Define the function $\sigma:U\to V$ by $\sigma(u)=S(u)\cap V$.
We claim that $\sigma$ is ${\cal C}^1$ in a neighborhood of $\ou$. To justify it, remember from the proof of (a) that $G$ is strictly proto-differentiable at $\ox$ for $\ou$.
Appealing to Theorem~\ref{prot} allows us to ensure that $G$ enjoys the same property at $x$ for $u$ whenever $(x,u)\in \gph G$ is sufficiently close to $(\ox,\ou)$.
Shrinking $U$ and $V$ if necessary, we can assume without loss of generality that $G$ is strictly proto-differentiable at $x$ for $u$ for any $(x,u)\in (\gph G)\cap  (V\times U)$.
Take any such a pair $(x,u)$ and conclude from $S^{-1}=G$ that  $S^{-1}$ is strictly proto-differentiable  at $u$ for $x$. Since strict proto-differentiability is a local property of $\gph S$ and since $\gph \sigma=(\gph S)\cap (U\times V)$, $\sigma$ is strictly proto-differentiable at $u$ for $x$. Employing \cite[Proposition~3.1]{r85} tells us that $\sigma$ is strictly differentiable on $U$, which is 
equivalent to saying that  $\sigma$ is ${\cal C}^1$ on $U$ due to \cite[Exercise~1D.8]{DoR14}.

To justify the claimed formula for the Jacobian matrix of $\sigma$ at $\ou$, take $p\in \R^n$   and set   $w=\nabla \sigma(\ou)p=D\sigma(\ou)(p)$. This implies that 
$p\in DG(\ox,\ou)(w)$. By  \eqref{calgd}, the latter means that $p-Aw\in \K^\perp$ and  $w\in \K$.  Recall that  $\K=\rge B$ and  $\K^\perp=\ker B^*$.
Thus, we find $q\in \R^s$ such that $w=Bq$, which leads us to 
$$
0=B^*(p-Aw)=B^*p-B^*ABq.
$$
Since the matrix $B^*AB$ is nonsingular, we get 
$$
\nabla \sigma(\ou)p=w=Bq=B(B^*AB)^{-1}B^*p,
$$
which completes the proof of theorem. 
\end{proof}

The  equivalence of metric regularity and strong metric regularity for generalized equations was first achieved in  \cite[Theorem~3]{dr96} by Donchev and Rockafellar  without 
assuming  the nondegeneracy for  solutions.  However, the framework in \cite{dr96} is associated with $\ph=\dd_C$    in \eqref{GE},   $C$ being a polyhedral convex set, 
and thus is narrower than the generalized equation in \eqref{GE}. Moreover, the  approach therein relies heavily on Robinson's results in \cite{rob3} and did not utilize strict proto-differentiability. The new approach of dealing with the latter equivalence was first presented in \cite{hjs}, where the main attention was given to polyhedral functions. 

\begin{Remark} \label{ghpk} A closer look into the proof of Theorem~\ref{mrsmr} reveals that the underlying property that allows us to achieve the equivalence of 
metric regularity and strong metric regularity  for the generalized equation in \eqref{GE} is strict proto-differentiability of the mapping $G$ in \eqref{mapg} at $\ox$ for $\ou$, which is 
indeed equivalent to the same property of the subgradient mapping $\sub \ph$ at $\ox$ for $\ou-f(\ox)$. As shown in Theorem~\ref{prot}, the latter amounts to the nondegeneracy of the  solution $\ox$
to \eqref{GE}. Since strict proto-differentiability is a local property of $\gph \sub \ph$, we could obtain the same equivalence  if $\sub \ph$ in \eqref{GE} is replaced with a graphical localization of $\sub \ph$ around $(\ox,\ou-f(\ox))$.
To elaborate more, suppose that $T:\R^n\tto\R^n$ is a set-valued mapping  such that $\gph T$ and $\gph \sub \ph$ coincide locally around $(\ox,\ou-f(\ox))$. In this case, if all the assumptions in
Theorem~\ref{mrsmr} are satisfied, then both conclusions  (a) and (b) in the latter result hold for the generalized equation 
$$
\ou\in f(x)+T(x).
$$ 
This observation will be utilized later in this section while we are proving continuous differentiability of the proximal mapping of $\ph$. 
\end{Remark}

An interesting case of   the generalized equation \eqref{GE} occurs when the composite function $\ph$ therein is the indicator function of a ${\cal C}^2$-smooth manifold $\Th$, 
defined by \eqref{CF2} with $C=\{0\}$. In this case, the relative interior condition in \eqref{nds} holds automatically because $\sub \ph(\ox)=\rge \nabla \Phi(\ox)^*$, which is a 
linear subspace of $\R^n$ and thus is  relatively  open. This tells us that every solution to \eqref{GE} with $\ph=\dd_\Th$ is nondegenerate. The next result shows that 
metric regularity and strong metric regularity in this case are always equivalent. 
\begin{Corollary} Assume that $\Th$ is  a ${\cal C}^2$-smooth manifold around $\ox$ in $\R^n$ and that  $\ox$ is a solution to the generalized equation in \eqref{GE}, where $\ph=\dd_\Th$.
Then the mapping $G$ in \eqref{mapg} is metrically regular at $\ox$ for $\ou$ if and only if it is strongly metrically regular at $\ox$ for $\ou$.
\end{Corollary}
\begin{proof} This results from our  discussion prior to this corollary and Theorem~\ref{mrsmr}(a).
\end{proof}

Our final task is to study continuous differentiability   of the proximal mapping of the composite function $\ph$ in \eqref{CF}. Recall that proximal mapping 
of $\ph$ for a parameter value  $r>0$,  denoted by  $\prox_{r\ph}$, is  defined  by 
\begin{equation*}\label{proxmap}
\prox_{r \ph}(x)= \argmin_{w\in \R^n}\Big\{\ph(w)+\frac{1}{2r}\|w-x\|^2\Big\}.
\end{equation*}
To ensure that $\prox_{r\ph}$ is nonempty, we have to assume that $\ph$ is prox-bounded, meaning that for some $\al\in \R$, the function $\ph+\al\|\cdot\|^2$ is bounded from below on $\R^n$; see \cite[Exercise~1.24]{rw} for more equivalent 
descriptions. An important instance of $\ph$
for which this condition automatically holds is when $\ph=\dd_\Th$ with  $\Th$   taken from \eqref{CF2}. In this case, the proximal mapping $\prox_{r\ph}$ boils down to the projection mapping 
$P_\Th$. Below, we collect some important properties of the proximal mapping of $\ph$, which will be utilized in our characterization of continuous differentiability of $\prox_{r\ph}$ in the rest of this section. 

\begin{Lemma}\label{proj1}
 Assume that $\ph$ has the representation \eqref{CF} around $\ox\in \R^n$ and  $\ov\in \sub \ph(\ox)$, and that the basic constraint qualification \eqref{bcq1} holds at $\ox$.
 Suppose further that $\ph$ is prox-bounded. 
 Then there exist positive constants $\ve$ and $\rho$ such that for any $r\in (0,1/\rho)$, there is a neighborhood $U_r$ of $\ox+r\ov$ on which $\prox_{r\ph}$ is nonempty, single-valued,   Lipschitz continuous and 
 can be calculated by 
\begin{equation}\label{proj3}
\prox_{r\ph}=(I+rT_\ve)^{-1},
\end{equation}
where  the set-valued mapping $T_\ve:\R^n\tto \R^n$ is defined by 
\begin{equation}\label{logh}
T_\ve(x)=\begin{cases}
\sub \ph(x)\cap \B_{\ve}(\ov)&\mbox{if}\;\; x\in \B_\ve(\ox),\\
\emptyset& \mbox{otherwise},
\end{cases}
\end{equation}
and where  $I$ stands for the $n\times n$ identity matrix. Moreover, we have 
\begin{equation}\label{mepr}
\nabla (e_r\ph)(x)=\frac{1}{r}(x-\prox_{r\ph}(x)), \;\; x\in U_r,
\end{equation}
where the Moreau envelope function   $e_{r}\ph$ is defined by 
$$
e_r\ph(x)=\inf_{w\in \R^n}\big\{\ph(w)+\frac{1}{2r}\|w-x\|^2\big\}, \;\; x\in \R^n.
$$
%   \begin{itemize}[noitemsep,topsep=2pt]
%\item [ \rm {(a)}]  For every $(x,v)\in \gph T_\ve$ and any $x'\in \R^n$, we have 
%$$
%\ph(x')\ge \ph(x)+\la v,x'-x\ra-\frac{\rho}{2}\|x'-x\|^2.
%$$
%\item [ \rm {(b)}] The mapping $T_{\ve}+\rho I$ is monotone, meaning that 
%$$
%\la v_1-v_0,x_1-x_0\ra\ge -\rho \|x_1-x_0\|^2 \quad \mbox{when}\;\; v_i\in T_\ve(x_i), \; i=0,1.
%$$
 
%\item [ \rm {(b)}] The  {\rm(}strong{\rm)} metrically regular of $\sub \ph$ at $\ox$ for $\ov$ is equivalent to that of $T_\ve$ at $\ox$ for $\ov$.

%\item [ \rm {(e)}] If $\ov\in \ri \rs   \ph(\ox)$, then $T_\ve$ is strictly proto-differentiable at $\ox$ for $\ov$.
\end{Lemma}
\begin{proof} By Proposition~\ref{phpro}(a), the composite function $\ph$ is prox-regular and subdifferentially continuous at $\ox$ for $\ov$.
Both   results about $\prox_{r\ph}$ were established in \cite[Theorem~4.4]{pr}.   
\end{proof}

We proceed with a  characterization  of  continuous differentiability of  the proximal mapping  of the composite function $\ph$ in \eqref{CF}.

\begin{Theorem}\label{prox1}
 Assume that $\ph$ has the representation \eqref{CF} around $\ox\in \R^n$ and  $\ov\in \sub \ph(\ox)$, and  that the SOQC in \eqref{soqc} holds at $(\ox,\olm)$, where $\olm$ is the unique vector in $\Lm(\ox,\ov)$.
 Suppose further that $\ph$ is prox-bounded. Then the following properties are equivalent:
  \begin{itemize}[noitemsep,topsep=2pt]
\item [ \rm {(a)}]  there exists a positive constant   $\rho$ such that for any $r\in (0,1/\rho)$, the proximal mapping $\prox_{r\ph}$ is ${\cal C}^1$ in a neighborhood of $\ox+r\ov$;
\item [ \rm {(b)}]  there exists a positive constant   $\rho$ such that for any $r\in (0,1/\rho)$, the envelope function   $e_{r}\ph$ is ${\cal C}^2$ in a neighborhood of $\ox+r\ov$;
\item [ \rm {(c)}] $\ov\in \ri \sub \ph(\ox)$.
\end{itemize}
\end{Theorem}
\begin{proof} Assume first that (c) holds and take positive constants $\ve$ and $\rho$ from Lemma~\ref{proj1}.  Fix $r\in (0,1/\rho)$ and 
consider  the generalized equation 
\begin{equation}\label{GE3}
\ou\in (I+rT_\ve)(x), \;\; x\in \R^n,
\end{equation}
where $\ou=\ox+r\ov$ and $T_\ve$ is defined by \eqref{logh}.  We know from  Lemma~\ref{proj1} that the solution set to \eqref{GE3}, namely $(I+rT_\ve)^{-1}(\ou)$, is single-valued  and is precisely $\prox_{r\ph}(\ou)$.
Moreover, we can find a neighborhood $U_r$ of $\ou$ such that for any $u\in U_r$ the solution mapping to the canonical perturbation of \eqref{GE3}, namely the  generalized equation
\begin{equation}\label{GE4}
u\in (I+rT_\ve)(x), \;\; x\in \R^n,
\end{equation}
is $\prox_{r\ph}(u)$. Also, $\prox_{r\ph}$ is single-valued and Lipschitz continuous on $U_r$. Thus, it follows from \eqref{proj3} and \cite[Proposition~3G.1]{DoR14}
that   $I+rT_\ve$ is metrically regular at $\ox$ for $\ou$.
We know from Theorem~\ref{prot} that $\sub \ph$ is strictly proto-differentiable at $\ox$ for $\ov$. This amounts to saying that $\rt_{\gph \sub \ph}(\ox,\ov)= \widetilde{T}_{\gph \sub \ph}(\ox,\ov)$.
By the definition of $T_\ve$ in \eqref{logh}, we obtain $\gph T_\ve=(\gph \sub \ph)\cap (\B_\ve(\ox)\times \B_\ve(\ov))$, meaning that 
 that the graphs of $T_\ve$ and $  \sub \ph$ coincide locally around $(\ox,\ov)$. This  implies that  $\rt_{\gph T_\ve}(\ox,\ov)= \widetilde{T}_{\gph T_\ve}(\ox,\ov)$, which 
means  that $T_\ve$ is strictly proto-differentiable at $\ox$ for $\ov$. This clearly shows that $rT_\ve$ is strictly proto-differentiable at $\ox$ for $r\ov$ and so is $I+rT_\ve$ at $\ox$ for $\ou$.
The latter results from the identity mapping being strict differentiable coupled with the sum for strict proto-differentiability, established in \cite[Proposition~5.3]{hjs}.  
As discussed in Remark~\ref{ghpk}, this allows us to obtain both conclusions (a) and (b) in Theorem~\ref{mrsmr} for the generalized equation in \eqref{GE3}. Thus, 
 $I+rT_\ve$ is metrically regular at $\ox$ for $\ou$ if and only if  the solution mapping to \eqref{GE4} 
has a Lipschitz continuous localization  around $\ou$ for $\ox$, which 
is  ${\cal C}^1$ in a neighborhood of $\ou$.  
Since we already showed that  $I+rT_\ve$ is metrically regular at $\ox$ for $\ou$, and since the solution mapping to \eqref{GE4} is $\prox_{r\ph}$ on $U_r$,
 we can conclude by shrinking $U_r$ if necessary that $\prox_{r\ph}$ is   ${\cal C}^1$ in a neighborhood of $\ou$. This completes the proof of the implication (c)$\implies$(a).

To justify the opposite implication, assume that (a) holds. Pick $r\in (0,1/\rho)$ and set  $h:=\prox_{r\ph}$. Since $h$ is ${\cal C}^1$ in a neighborhood of $\ox+r\ov$,  
we conclude from \cite[Proposition~3.1]{r85} that $T_{\gph h}(\ox+r\ov, \oy)$  with     $\oy:=h(\ox+r\ov)$ is a linear subspace. It follows from \eqref{proj3}  that 
\begin{equation}\label{gdpr}
Dh(\ox+r\ov)=\big(I+rD(\sub \ph)(\ox,\ov)\big)^{-1};
\end{equation}
see the proof of \cite[Exercise~12.64]{rw} for a similar result. 
This formula implies that 
$$
(w,q)\in   T_{\gph \sub \ph}(\ox,\ov)\iff (w+rq,w)\in T_{\gph h}(\ox+r\ov, \oy).
$$
Thus $T_{\gph \sub \ph}(\ox,\ov)$ is a linear subspace since $T_{\gph h}(\ox+r\ov, \oy)$  has this property. 
By $T_{\gph \sub \ph}(\ox,\ov)=\gph D(\sub \ph)(\ox,\ov)$, we conclude that $\dom D(\sub \ph)(\ox,\ov)$ is a linear subspace. 
We know from  \cite[Theorem~7.2]{mms0} that $\dom D(\sub \ph)(\ox,\ov)=K_{\ph}(\ox,\ov)$
and from \eqref{criph}  that $K_{\ph}(\ox,\ov)=N_{\sub \ph(\ox)}(\ov)$. This, combined with $K_{\ph}(\ox,\ov)$  being a linear subspace, tells us that $\ov\in \ri \sub \ph(\ox)$
and hence proves (c).
The equivalence of (a) and (b) results from the identity in \eqref{mepr}. This completes the proof.
\end{proof}

It is important to mention that a characterization of continuous differentiability of the proximal mapping of prox-regular functions, which encompass the composite function $\ph$ in \eqref{CF},
was achieved in \cite[Theorem~4.4]{pr2} under a rather restrictive assumption that $\ox$, in the notation of Theorem~\ref{prox1}, must be a global minimum of $\ph$. This forces us to just 
deal with the case that $\ov=0$.  Our new approach, which stems from the equivalence of metric regularity and strong metric regularity for generalized equations at their nondegenerate solutions, 
allows us to drop the latter restriction. Note also that the given characterization in \cite{pr2} does not go far enough to provide a characterization of  continuous differentiability of the proximal mapping
via the relative interior condition in Theorem~\ref{prox1}(c).

Continuous differentiability of the projection mapping to prox-regular sets in Hilbert spaces  was studied recently in \cite{cst} by extending the approach, pioneered  by Holmes in \cite{hol} 
for convex sets in Hilbert spaces. His main result, \cite[Theorem~2]{hol}, states 
that if $\Omega\subset \R^d$ is a closed convex set, $x\in \R^d$,  the boundary of  $\Omega$ is a ${\cal C}^2$ smooth manifold  around $y=P_\Omega(x)$, then the projection mapping $P_\Omega$ is ${\cal C}^1$ in a neighborhood of the open normal ray 
$\{y+t(x-y)|\; t>0\}$. It was shown in \cite[Theorem~2.4]{cst} that under a similar assumption with convexity replaced by prox-regularity, one can find $\rho>0$ such that  a similar result holds for 
 $P_\Omega$  in a neighborhood of the open normal ray $\{y+t(x-y)|\; t\in (0,1/\rho)\}$. Holmes's result in \cite{hol} and its extension in \cite{cst} require that the boundary of the set under consideration be 
 a ${\cal C}^2$ smooth manifold,   a condition which is not needed in our approach. Moreover, we were able to provide a characterization of continuous differentiability of the proximal mapping of functions, which goes beyond the projection mapping 
 onto sets. The price we paid is that our framework is narrower than \cite{cst} and that  we assume the space under consideration is finite dimensional. As we show in our subsequent paper, the main driving force behind finding a 
 characterization of continuous differentiability of prox-regular sets in a finite dimensional space is strict proto-differentiability, which is equivalent to the relative interior condition in Theorem~\ref{prox1}(c) when we are dealing with a set with the representation \eqref{CF2}.
 
{  We should also add here that Shapiro in \cite[Proposition~3.1]{sh16} characterized the differentiability of projection onto ${\cal C}^2$-reducible convex sets in the sense of \cite[Definition~3.135]{bs} 
 using a similar relative interior condition in Theorem~\ref{prox1}(c). Shapiro's result was generalized in \cite[Lemma~5.3.32]{mi} for ${\cal C}^2$-decomposable convex functions in the sense of \cite{sh03}.
 Note that while any function $\ph$  satisfying the assumptions in Theorem~\ref{prox1} is ${\cal C}^2$-decomposable (cf. \cite[Example~5.3.17]{mi} and \cite[Lemma~5.3.23]{mi}), 
   Theorem~\ref{prox1} can't be covered by \cite[Lemma~5.3.32]{mi}, since the latter result   can only be applied   to convex functions. Moreover, it doesn't go far enough to 
   characterize continuously differentiability as   Theorem~\ref{prox1}. Finally, we should point out that \cite[Theorem~28]{dhm} demonstrates that the proximal mapping of any ${\cal C}^2$-partly smooth function (cf.  
   of \cite[Definition~14]{dhm})  satisfying the relative interior condition in Theorem~\ref{prox1}(c) is ${\cal C}^1$. As pointed above, any   function $\ph$  satisfying the assumptions in Theorem~\ref{prox1} 
   is ${\cal C}^2$-decomposable. The latter class of functions was shown by Shapiro in \cite{sh03} to be ${\cal C}^2$-partly smooth  when the relative interior condition holds. Combining these 
   indicates that the implication (a)$\implies$(c) in  Theorem~\ref{prox1} can be   obtained from   \cite[Theorem~28]{dhm}. }

Next, we are going to present a characterization of continuous differentiability of the projection mapping to $\Th$, where $\Th$ is defined by \eqref{CF2}. 
In this case, it is not hard to see from  \eqref{code2}  that the SOQC in \eqref{soqc} boils down to 
\begin{equation}\label{simso}
\spann\big\{N_C(\Phi(\ox))\big\} \cap \ker \nabla \Phi(\ox)^*=\{0\}
\end{equation}
with the polyhedral convex set $C$ taken from \eqref{CF2}.

\begin{Corollary}\label{projsm}
Assume that $\Th$ is a subset of $\R^n$ and $\ox\in \Th$ with the representation \eqref{CF2} around $\ox$ and that 
$\ov\in N_\Th(\ox)$ and  the condition \eqref{simso} holds at $\ox$. Then the following conditions are equivalent:
\begin{itemize}[noitemsep,topsep=2pt]
\item [ \rm {(a)}]   there exists a positive constant   $\rho$ such that for any $r\in (0,1/\rho)$, the projection mapping $P_\Th$ is ${\cal C}^1$ in a neighborhood of $\ox+r\ov$;
\item [ \rm {(b)}]   there exists a positive constant   $\rho$ such that for any $r\in (0,1/\rho)$,  the  function $x\mapsto \dist(x, \Th)^2$ is  ${\cal C}^2$  in a neighborhood of $\ox+r\ov$;
\item [ \rm {(c)}]  $\ov\in \ri N_\Th(\ox)$.
\end{itemize}
\end{Corollary} 

\begin{proof}  This results from Theorem~\ref{prox1} by setting $\ph=\dd_\Th$ therein. Note that $\dd_\Th$ is clearly prox-bounded, which is required in Theorem~\ref{prox1}. 
\end{proof}

An important example of the set $\Th$ with the representation \eqref{CF2} around $\ox$ that satisfies \eqref{simso} 
 is   ${\cal C}^2$-smooth manifolds. This happens 
when the polyhedral convex set $C$ in \eqref{CF2} is $\{0\}$ with $0\in \R^m$.  Observe from Example~\ref{point}(b) that   the condition \eqref{simso} 
amounts to $\nabla \Phi(\ox)$ having full rank when $\Th$ is a ${\cal C}^2$-smooth manifold. In this case, the relative interior condition  $\ov\in \ri N_\Th(\ox)$ automatically 
holds since we have $N_\Th(\ox)=\rge \nabla \Phi(\ox)^*$. By Corollary~\ref{projsm},  we can find a positive constant   $\rho$ such that for any $r\in (0,1/\rho)$, the projection mapping $P_\Th$
 is always ${\cal C}^1$ in a neighborhood of $\ox+r\ov$ whenever $\ov\in  N_\Th(\ox)$.  In particular, when $\ov=0$, we deduce from \eqref{gdph}
and Proposition~\ref{phpro}(b)  that  
 $$
 DN_\Th(\ox,\ov)= N_{K_\Th(\ox,\ov)}=N_{T_\Th(\ox)},
 $$
where the last equality results from $\ov=0$ and the definition of the critical cone of $\Th$. This, together with \eqref{gdpr}, tells us that 
$$
\nabla P_\Th(\ox)=DP_\Th(\ox)=(I+N_{T_\Th(\ox)})^{-1}=P_{T_\Th(\ox)}.
$$ 
This formula and continuous differentiability of $P_\Th$ for a ${\cal C}^2$ smooth manifold $\Th$ at any point $\ox\in \Th$ were observed  before in \cite[Lemma~2.1]{lm}, which is 
a special case of Corollary~\ref{projsm} for a ${\cal C}^2$ smooth manifold $\Th$ with $\ov=0$.
 
\section*{Acknowledgments}
The first author would like to thank Vietnam Institute for Advanced Study in Mathematics for hospitality during her post-doctoral fellowship of the Institute in 2021--2022.
The references \cite{bm, dhm, stp17, stp18} were brought to our attention by  two anonymous referees that is highly appreciated.

%%%%%%%%%%%%%%%%%%%%%%%%%%%%
%%%%%%%%%%%%%%%%%%%%%%%%%%%%%%%%%%%%


\begin{thebibliography}{99} 

\bibitem{bm} {\sc M. Benko and P. Mehlitz}, {\em  Why second-order sufficient conditions are, in a way, easy – or – revisiting calculus for second subderivatives},  arXiv:2206.03918, 2022. 

%\bibitem{bgy} {\sc M. Benko, H. Gfrerer, J. J. Ye, J. Zhang, and J. Zhou}, {\em Second-order optimality conditions for general nonconvex optimization problems and variational analysis of disjunctive systems},
% arXiv:2203.10015, 2022.

\bibitem{be} {\sc J. V. Burke and A.  Engle}, {\em Strong metric (sub)regularity of KKT mappings for piecewise linear-quadratic convex composite
optimization},   Math. Oper. Res.,  45 (2020), pp. 797--1192.

\bibitem{bs} {\sc J. F. Bonnans and A. Shapiro}, {\em Perturbation Analysis of Optimization Problems}, Springer, New York, 2000.



\bibitem{cst} {\sc  R. Correa, D. Salas, and L. Thibault}, {\em Smoothness of the metric projection onto nonconvex bodies in Hilbert spaces}, J. Math. Anal. Appl., 457 (2018), pp. 1307--1332.

\bibitem{dhm} {\sc A. Daniilidis, W. Hare, and J. Malick}, {\em Geometrical interpretation of the predictor-corrector type algorithms in structured optimization problems}, Optimization, 55 (2006), pp. 481–503.

\bibitem{DoR14} {\sc A. L. Dontchev and R. T. Rockafellar}, {\em Implicit Functions and Solution Mappings: A View from Variational Analysis}, Springer, New York, 2014.

\bibitem{dr96} {\sc A. L. Dontchev and R. T. Rockafellar}, {\em Characterizations of strong regularity for variational inequalities over polyhedral convex sets}, SIAM J. Optim., 6 (1996), pp. 1087--1105.

\bibitem{hjs} {\sc  N. T. V. Hang, W. Jung, and M. E. Sarabi}, {\em Role of subgradients in variational analysis of polyhedral functions}, arXiv:2207.07470v2, 2022.

\bibitem{hs} {\sc N. T. V. Hang and  M. E. Sarabi},  {\em Local convergence analysis of augmented Lagrangian methods for piecewise linear-quadratic composite optimization problems},  SIAM J. Optim., 31 (2021), pp. 2665--2694.

\bibitem{hs23} {\sc N. T. V. Hang and  M. E. Sarabi},  {\em  Strict proto-differentiability and its applications in parametric optimization},  preprint, 2023.

\bibitem{hol}  {\sc R. B. Holmes}, {\em Smoothness of certain metric projections on Hilbert space}, Trans. Amer. Math. Soc., {  183} (1973), pp. 87--100.

 \bibitem{io} {  \sc A. D. Ioffe}, {  \em Variational Analysis of Regular Mappings: Theory and Applications}. Springer, Cham, Switzerland, 2017.
 
 
  \bibitem{lm} {\sc A. S. Lewis and J. Malick}, {\em Alternating projections on manifolds}, Math. Oper. Res., 33 (2008), pp. 216--234.
 
  \bibitem{lw} {\sc A. S. Lewis and C. Wylie}, {\em Active-set Newton methods and partial smoothness}, Math. Oper. Res., 46 (2021), pp. 712--725.


\bibitem{mi} {\sc A.  Milzarek}, {\em Numerical Methods and Second Order Theory for Nonsmooth Problems}, Ph.D. Dissertation, University of Munich, 2016.

\bibitem{mms}  {\sc  A. Mohammadi, B. S. Mordukhovich, and M. E. Sarabi}, {\em   Parabolic regularity via geometric variational analysis}, Trans. Amer. Soc., 374 (2021), pp. 1711--1763.

\bibitem{mms0} {\sc  A. Mohammadi, B. S. Mordukhovich, and M. E. Sarabi}, {\em Variational analysis of composite models with applications to continuous optimization}, Math. Oper. Res., 47 (2022), pp. 379--426.

\bibitem{ms20} {\sc A. Mohammadi and M. E. Sarabi},  {\em Twice epi-differentiability of extended-real-valued functions with applications in
composite optimization},  SIAM J. Optim., 30 (2020), pp. 2379--2409.

\bibitem{m06} {\sc B. S. Mordukhovich},  {\em Variational Analysis and Generalized Differentiation, I: Basic Theory; II: Applications}. Springer Berlin, Heidelberg, 2006.

\bibitem{mr} {\sc B. S. Mordukhovich and R. T. Rockafellar}, {\em Second-order subdifferential calculus with application to tilt stability in optimization}, {SIAM
J. Optim.,} {  22} (2012), pp. 953--986.

\bibitem{msr} {\sc B. S. Mordukhovich, R. T. Rockafellar, and M. E. Sarabi}, {\em Characterizations of full stability in constrained optimization}, SIAM J. Optim., 23 (2013), pp. 1810--1849.

%\bibitem{MoS16} {\sc B. S. Mordukhovich and M. E. Sarabi}, {\em Generalized differentiation of piecewise linear functions in second-order variational analysis}, Nonlinear Anal., 132 (2016), pp.  240--273.

\bibitem{ms16} {\sc B. S. Mordukhovich and M. E. Sarabi}, {\em Second-order analysis of piecewise linear functions with applications to optimization and stability}, J. Optim. Theory Appl., 171 (2016),  pp. 504--526.

\bibitem{ms21} {\sc B. S. Mordukhovich and M. E. Sarabi}, {\em Generalized Newton algorithms for tilt-stable minimizers in nonsmooth optimization},  SIAM J. Optim., 31 (2021), pp. 1184--1214.


\bibitem{pr3} {\sc R. A. Poliquin and R. T. Rockafellar}, {\em Second-order nonsmooth analysis in nonlinear programming}, in Recent Advances in Optimization (D. Du, L. Qi and R. Womersley, eds.), World Scientific Publishers, 1995, pp. 322--350.

\bibitem{pr} {\sc R. A. Poliquin and R. T. Rockafellar}, {\em Prox-regular functions in variational analysis } Trans. Amer. Math. Soc., 348 (1996), pp. 1805--1838.

\bibitem{pr2} {\sc R. A. Poliquin and R. T. Rockafellar}, {\em Generalized Hessian properties of regularized nonsmooth functions}, SIAM J. Optim., 6 (1996), pp. 1121--1137.

\bibitem{pr1} {\sc R. A. Poliquin and R. T. Rockafellar}, {\em Tilt stability of a local minimum}, SIAM J. Optim., 8 (1998), pp. 287--299.

%\bibitem{rob2} {\sc  S. M. Robinson}, {\em Local structure of feasible sets in nonlinear programming. Part II: Nondegeneracy}, Math. Program. Studies 22 (1984), pp. 217--230.

%\bibitem{rob} {\sc  S. M. Robinson}, {\em   An implicit-function theorem for a class of nonsmooth functions}, {  Math. Oper. Res.},  16 (1991), pp. 292--309.

\bibitem{rob3} {\sc S. M. Robinson,} {\em Normal maps induced by linear transformations}, Math. Oper. Res., {  17} (1992), pp. 691--714.

%\bibitem{Roc70} {\sc R. T. Rockafellar}, {\em Convex Ananlysis}, Priceton University Press, Princeton, New Jersey, 1970.

\bibitem{rw} {\sc R. T. Rockafellar and R. J-B. Wets}, {\em Variational Analysis}, Springer, Berlin, 1998.

\bibitem{r85} {\sc R. T. Rockafellar}, {\em Maximal monotone relations and the second derivatives of nonsmooth functions}, Ann. Inst. H. Poincar\'e  Analyse Non Lin\'eaire,  2 (1985), pp. 167--184. 

\bibitem{r88} {\sc R. T. Rockafellar}, {\em First- and second-order epi-differentiability in nonlinear programming}, Trans. Amer. Math. Soc., 307 (1988),  pp. 75--108.

\bibitem{r89} {\sc R. T. Rockafellar},  {em Generalized second derivatives of convex functions and saddle functions}, Trans. Amer. Math. Soc., 322 (1990), pp. 51--77.

\bibitem{r17} {\sc R. T. Rockafellar}, {\em Variational convexity and the local monotonicity of subgradient mappings}, Vietnam Journal of Math., 47 (2019), pp. 547--561.

\bibitem{r22} {\sc R. T. Rockafellar},  {\em Augmented Lagrangians and hidden convexity in sufficient conditions for local optimality}, Math. Program., 198 (2023), pp. 159--194.

\bibitem{s20}  {\sc M. E. Sarabi}, {\em Primal superlinear convergence of SQP methods for piecewise linear-quadratic composite optimization problems}, Set-Valued Var. Anal., 30  (2022),  pp. 1--37.

\bibitem{sh03}  {\sc A. Shapiro},  {\em On a class of nonsmooth composite functions},  Math. Oper. Res., 28 (2003), pp. 677--692.

\bibitem{sh16}  {\sc A. Shapiro},  {\em Differentiability properties of metric projections onto convex sets}, J. Optim. Theory Appl., 169 (2016), pp. 953--964.

\bibitem{stp17}  {\sc  L. Stella, A. Themelis, and P. Patrinos}, {\em Forward-backward quasi-Newton methods for nonsmooth optimization problems}, Comput. Optim. Appl., 67 (2017), pp. 443--487.

\bibitem{stp18}  {\sc  L. Stella, A. Themelis, and P. Patrinos}, {\em Forward-backward envelope for the sum of two nonconvex functions: further properties and nonmonotone linesearch algorithms}, SIAM J. Optim., 28 (2018), pp. 2274--2303.


\end{thebibliography}
\end{document}